\documentclass[11pt]{article} 

\usepackage[utf8]{inputenc} 

\usepackage{geometry} 
\usepackage{graphicx} 

\usepackage{booktabs} 
\usepackage{array} 
\usepackage{paralist} 
\usepackage{verbatim} 
\usepackage{subfig} 
\usepackage{xcolor}

\usepackage{fancyhdr} 
\usepackage{amsmath,amssymb,amstext,amsthm}
\usepackage{sectsty}
\allsectionsfont{\sffamily\mdseries\upshape} 

\usepackage[nottoc,notlof,notlot]{tocbibind} 
\usepackage[titles,subfigure]{tocloft} 

\newtheorem{theorem}{\indent Theorem}[section]
\newtheorem{lemma}[theorem]{\indent Lemma}
\newtheorem{remark}[theorem]{\indent Remark}
\newtheorem{proposition}[theorem]{\indent Proposition}
\newtheorem{defn}{Definition}[section]

\renewcommand{\l}{\langle}
\renewcommand{\r}{\rangle}
\newcommand{\dive}{{\mathrm{div}\,}}
\newcommand{\dd}{{\mathrm{d}}}
\newcommand{\uu}{{\mathbf{u}}}
\newcommand{\ff}{{\mathbf{f}}}
\newcommand{\vg}{{\mathbf{g}}}
\newcommand{\UU}{{\mathbf{U}}}
\newcommand{\FF}{{\mathbf{F}}}
\newcommand{\GG}{{\mathbf{G}}}

\newcommand{\RR}{{\mathbb{R}}}
\newcommand{\e}{\varepsilon}

\newcommand{\p}{\partial}
\newcommand{\aaa}{\alpha}
\newcommand{\calB}{\mathcal{B}}

\newcommand{\uue}{\widetilde \uu_{\e}}

\newcommand\bp{\begin{pmatrix}}
\newcommand\ep{\end{pmatrix}}
\newcommand\be{\begin{equation}}
\newcommand\ee{\end{equation}}
\newcommand\ba{\begin{equation}\begin{aligned}}
\newcommand\ea{\end{aligned}\end{equation}}

\newcommand\nn{\nonumber}

\usepackage{layout}

\numberwithin{equation}{section}

\title{Homogenization of evolutionary incompressible Navier-Stokes system in perforated domains}
\author{Yong Lu\footnote{Department of Mathematics, Nanjing University, Nanjing 210093, China, luyong@nju.edu.cn} \and Peikang Yang\footnote{Department of Mathematics, Nanjing University, Nanjing 210093, China, ypk@smail.nju.edu.cn}}
\date{} 

\begin{document}

\maketitle

\centerline{\em{This paper is dedicated to the memory of  Anton\'in Novotn\'y.}}

\begin{abstract}
In this paper, we consider the homogenization problems for evolutionary incompressible Navier-Stokes system in three dimensional domains perforated with a large number of small holes which are periodically located. We first establish certain uniform estimates for the weak solutions.  To overcome the extra difficulties coming from the time derivative, we use the idea of Temam \cite{ref5} and consider the equations by integrating in time variable. After suitably extending the weak solutions to the whole domain, we employ the generalized cell problem  to study the limit process. 
\end{abstract}


\section{Introducton}
\subsection{Background}
In this paper we study the homogenization of homogeneous incompressible Navier-Stokes equations in a perforated domain in $\mathbb{R}^{3}$ under Dirichlet boundary condition. Our goal is to describe the limit behavior of the (weak) solutions as the number of holes goes to infinity and the size of holes goes to zero simultaneously.

Let $\Omega  \subset {\mathbb{R}^3}$ be a bounded domain of class $C^{2,\beta}~(0<\beta<1)$. The holes in $\Omega$ are denoted by $T_{\varepsilon,k}$ which are assumed to satisfy
\begin{equation}\label{1-hole}
B(\varepsilon {x_k},{\delta _0}{a_\varepsilon}) \subset T_{\varepsilon ,k} = \varepsilon {x_k} + a_\varepsilon T_{0} \subset  \subset B(\varepsilon {x_k},{\delta _1}{a_\varepsilon})  \subset \subset B(\varepsilon {x_k},{\delta _2}{a_\varepsilon}) \subset B(\varepsilon {x_k}, \delta_3\varepsilon) \subset  \varepsilon {Q_k},
\end{equation}
where the cube $Q_k: = (-\frac{1}{2},\frac{1}{2})^3+k$ and $x_k=x_0+k$ with $x_0\in T_{0}$, for each $k\in \mathbb{Z}^3$; $T_{0}$ is a model hole which is assumed to be a closed bounded and simply connected  $C^{2,\beta}$ domain; $\delta _i, \ i = 0,1,2,3$ are fixed positive numbers. The perforation parameters $\varepsilon$ and $a_\varepsilon$ are used to measure the mutual distance of holes and the size of holes, and $\varepsilon x_k=\varepsilon x_0+\varepsilon k$ present the locations of holes. Without loss of generality, we assume that
$x_0=0$ and $0<a_\varepsilon \leqslant \varepsilon \leqslant 1$.

The perforated domain $\Omega_{\e}$ under consideration is described as follows:
\begin{equation}\label{1-domain}
\Omega _\varepsilon : = \Omega \backslash \bigcup\limits_{k \in K_\varepsilon} {T_{\varepsilon ,k}},~~~~{K_\varepsilon}:= \{ k \in \mathbb{Z}^3:  \varepsilon \overline {Q_k}  \subset \Omega \}.
\end{equation}

The study of homogenization problems in fluid mechanics have gained a lot interest.  In particular, the homogenization of Stokes system in perforated domains has been systematically studied. In 1980s, Tartar \cite{ref7} considered the case where the size of holes is proportional to the mutual distance of holes and derived Darcy's law. In 1990s,  Allaire \cite{ref3,ref4} considered general size of holes and obtained complete results in periodic setting. By introducing a local problem and employing an abstract framework the idea of which goes back to \cite{CM82},  Allaire found that the homogenized limit equations are determined by the ratio $\sigma_\varepsilon$ given as
\begin{equation}\label{1-sigma}
\sigma _\varepsilon: = \Big(\frac{\varepsilon^d}{a_\varepsilon^{d-2}}\Big)^{\frac{1}{2}},  \ d \geqslant 3;\quad{\sigma _\varepsilon}: = \varepsilon \left| \log \frac{{a_\varepsilon }}{\varepsilon} \right|^{\frac{1}{2}}, \ d = 2,
\end{equation}
where $d$ is the spatial dimension. More precisely, if ${\lim _{\varepsilon  \to 0}}{\sigma _\varepsilon }= 0$ corresponding to the case of large holes, the homogenized system is the
Darcy's law; if ${\lim _{\varepsilon  \to 0}}{\sigma _\varepsilon}= \infty$ corresponding to the case of small holes, the motion of the fluid does not change much in the homogenization process and in the limit there arise the same Stokes equations; if ${\lim _{\varepsilon  \to 0}}{\sigma _\varepsilon }= \sigma _* \in(0,\infty) $ corresponding to the case of critical size of holes, the homogenized system is governed by the Brinkman's law—a combination of the Darcy's law and the original Stokes equations.

The homogenization study is extended to more complicated models decribing fluid flows: Mikeli\'c  \cite{ref2} has studied the incompressible Navier-Stokes equations in a porous medium; Masmoudi \cite{ref12} studied the compressible Navier-Stokes equations; Feireisl, Novotn\'y and Takahashi \cite{ref13} studied the Navier-Stokes-Fourier equations. In all these previous studies, only the case where the size of holes is proportional to the mutual distance of  holes is considered and the Darcy's law is recovered in the limit. Recently, Feireisl, Namlyeyeva and Ne\v{c}asov\'a  \cite{ref14} studied the case with critical size of holes for the incompressible Navier-Stokes equations and they derived Brinkman's law; Feireisl et al. also considered the case of small holes for the compressible Navier-Stokes equations \cite{ref15,ref10,ref9}.

\medskip

In \cite{ref3,ref4}, Allaire also gave a rather complete description concerning the homogenization of stationary incompressible Navier-Stokes equations  and the results coincide with the Stokes equations: for the case of small holes, the equations remain unchanged; for the case of large holes, Darcy's law is derived; for the case of critical size of holes, Brinkmann type equations are recovered. While, for the evolutionary incompressible Navier-Stokes equations, the study is not compete with respect to the size of holes, even in periodic setting: in \cite{ref2}, Mikeli\'c considered the case when the size of holes is proportional to the mutual distance of  holes, and in \cite{ref14} the critical size of holes is considered.

In this paper we shall consider the homogenization of evolutionary incompressible Navier-Stokes equations in perforated domain $\Omega _\varepsilon$ with Dirichlet boundary condition and our goal is to give a complete description for the homogenization process related to small and large sizes of holes in periodic setting. Let $T>0$, the initial boundary problem in space-time cylinder $\Omega_{\e} \times (0,T)$ under consideration is the following:
\begin{equation}\label{1-equation}
\begin{cases}
\partial _t \uu_\varepsilon + \dive   (\uu_\varepsilon \otimes \uu_\varepsilon)
-\mu \Delta \uu_\varepsilon + \nabla {p_\varepsilon} = \ff_{\e}, &\mbox{in}\; \Omega _\varepsilon \times (0,T),\\
\dive  \uu_\varepsilon = 0, &\mbox{in}\;\Omega _\varepsilon \times (0,T),\\
\uu_\varepsilon = 0, &\mbox{on}\;\partial\Omega _\varepsilon \times (0,T),\\
\uu_\varepsilon |_{t = 0} = \uu_{\e}^0 \in L^2(\Omega_{\varepsilon};\mathbb{R}^3).
\end{cases}
\end{equation}
Here $\uu_{\e}$ is the fluid velocity field in $\RR^{3}$ and $p_{\e}$ is the fluid pressure. The external force $\ff_{\e}$ is assume to be in $L^{2}(\Omega_{\e}\times (0,T);\RR^{3})$.

We recall some  notations. Let $W^{1,q}_{0}(\Omega)$ be the collection of Sobolev functions in $W^{1,q}(\Omega)$ with zero trace, and let $W^{-1,q}$ be the dual space of $W^{1,q}_{0}(\Omega)$.  We set $V^{1,q}(\Omega):=\{v\in W_0^{1,q}(\Omega),\ \dive  v =0\}$ with $W^{1,q}$ norm and ${V^{-1,q}}(\Omega)$ be the dual space of ${V^{1,q}}(\Omega)$. Let $L^{q}_{0}(\Omega)$ be the collection of $L^{q}(\Omega)$ integrable functions that are of zero average. We sometimes use $L^r L^s$ to denote the Bochner space $L^r(0,T;L^s(\Omega))$ or $L^r(0,T;L^s(\Omega);\RR^{3})$ for short. For a function $g$ in $\Omega_{\e}$, we use the notation $\widetilde  g$ to represent its zero extension in $\Omega$:
$$
\widetilde  g = g \ \mbox{in $\Omega_{\e}$}, \quad \widetilde  g = 0 \  \mbox{in $\Omega\setminus \Omega_{\e} = \bigcup\limits_{k \in {K_\varepsilon }} {T_{\varepsilon ,k}}$}.
$$

We  now recall the definition of (finite energy) weak solutions:
\begin{defn}\label{def-weak}
We call  $\uu_\varepsilon$ a weak solution of \eqref{1-equation} in $\Omega_{\e}\times (0,T)$ provided:
\begin{itemize}
\item There holds:
\ba\label{def-weak-1}
& \uu_\varepsilon \in L^2(0,T;V^{1,2}(\Omega _\varepsilon))\cap C_{\rm weak}([0,T], L^{2}(\Omega_{\e}))\cap C([0,T], L^{q}(\Omega_{\e})),  \\
& \partial _t\uu_\varepsilon \in L^\frac{4}{3}}(0,T;{V^{ - 1,2}}({\Omega _\varepsilon) ),
\ea
for any $1\leq q <2$.

\item For any $\varphi\in C^\infty_c(\Omega_{\e} \times [0,T) ;\mathbb{R}^3)$ with $\dive  \varphi =0$, there holds
\ba\label{5}
& \int_0^T \int_{\Omega _\varepsilon} - \uu_\varepsilon  \cdot \partial_t\varphi - \uu_\varepsilon  \otimes \uu_\varepsilon :\nabla \varphi
+ \mu \nabla \uu_\varepsilon:\nabla \varphi   \dd x\dd t  \\
& = \int_0^T \int_{\Omega_\varepsilon} \ff_{\e}   \cdot \varphi \dd x\dd t + \int_{\Omega _\varepsilon} \uu_\e^0  \cdot \varphi (x,0) \,\dd x\dd t.
\ea
The pressure $p_{\e}$ is determined by
\ba\nn
\l \nabla p_{\e}, \psi \r = - \l \partial _t \uu_\varepsilon + \dive  (\uu_\varepsilon \otimes \uu_\varepsilon)
-\mu \Delta \uu_\varepsilon   - \ff_{\e}, \psi \r, \quad \forall \,\psi \in C_{c}^{\infty}(\Omega_{\e} \times (0,T);\RR^{3}).
\ea

We further call $\uu_{\e}$ a finite energy weak solution of \eqref{1-equation} provided there holds in addition the energy inequality: for a.a. $t\in  (0,T),$
\begin{equation}
\frac{1}{2}\int_{\Omega _\varepsilon} \left| \uu_\varepsilon (x,t) \right|^2 \dd x
+\int_0^t \int_{\Omega _\varepsilon} \left| \nabla \uu_\varepsilon(x,s) \right|^2 \dd x\dd s
\leqslant \int_0^t \int_{\Omega _\varepsilon} \ff_{\e} \cdot \uu_\varepsilon  \dd x\dd s
+ \frac{1}{2}\int_{\Omega _\varepsilon} \left| \uu^0_{\e}(x) \right|^2\dd x.\label{6}
\end{equation}

\end{itemize}
\end{defn}

For each fixed $\varepsilon$, the existence of a global finite energy weak solution $\uu_\varepsilon$ is known, see for example Leray's pioneer work \cite{ref18} or the classical books  \cite{ref6,ref5}. We need to investigate the behavior of the solutions as $\e \to 0.$

\subsection{Main Results}

%

From \eqref{1-sigma}, we see that the size of holes $a_\varepsilon$ is typically chosen to be $\varepsilon^\alpha~ (\alpha\geqslant 1)$ in three or higher dimensional spaces, while $a_\varepsilon$ is typically chosen to be $e^{-\varepsilon^{-\alpha}}~(\alpha>0)$ in two dimensional case.  Here for the study of three dimensional case, we shall take
\be\label{ae-se}
a_{\e} = \e^{\alpha}  \ \mbox{with} \ \alpha\geqslant 1, \ \mbox{which implies} \ \sigma_{\e} = \e^{\frac{3-\alpha}{2}}.
\ee

Throughout the paper, we will assume the zero extension of the initial datum and the external force satisfy
\be\label{ini-force}
\widetilde  \uu_{\e}^{0} \to \uu^{0}  \ \mbox{strongly in}  \ L^{2}(\Omega),\quad\widetilde  \ff_{\e} \to \ff \  \mbox{strongly in}  \ L^{2}(\Omega\times (0,T)).
\ee

Let $\uu_\varepsilon$ be a finite energy weak solution of \eqref{1-equation}.  We employ the idea of Mikeli\'c \cite{ref2} and Temam \cite{ref5} and introduce for any $t\in(0,T)$,
\begin{equation}\label{1-new-def}
\UU_\varepsilon(\cdot, t): =\int_0^t\uu_\varepsilon (\cdot, s) \dd s,\quad \Psi_{\e} (\cdot, t) :=  \int_0^t \uu_\varepsilon (\cdot, s) \otimes \uu_\varepsilon (\cdot, s) \dd s, \FF_\e(\cdot, t) : =\int_0^t\ff_\e (\cdot, s) \dd s.
\end{equation}
Then $ \UU_{\e} \in C([0, T];W^{1,2}_{0}(\Omega_\varepsilon),  {\div \UU}_\varepsilon=0, \  \FF_\e\in C([0, T];L^2(\Omega_\varepsilon))$ and $\div \Psi _\varepsilon\in C([0,T]; L^{\frac{3}{2}}(\Omega_\varepsilon))$. As shown in the proof of our theorems, instead of showing the limit behavior of $\uu_{\e}$, we turn to study the limit behavior of $\UU_{\e}$ which has better regularity in time variable. Clearly
\be\label{Fe-F-st}
\widetilde \FF_{\e} \to \FF  = \int_0^t\ff (s) \dd s \  \mbox{strongly in}  \ L^{2}(\Omega \times (0,T)).
\ee
The classical theory on Stokes equations implies that there exists $P_\varepsilon\in C([0,T];L^2_{0}(\Omega_\varepsilon))$ (see Chapter 3 in \cite{ref5}), such that for any $t\in (0,T)$,
\begin{equation}\label{2-equa}
\uu_\varepsilon(t) - \uu^0_\e + \dive\Psi _\varepsilon(t) - \mu \Delta \UU_\varepsilon (t)+\nabla P_\varepsilon(t) = \FF_\e(t),\quad\text{in}~W^{ - 1,2}(\Omega _\varepsilon).
\end{equation}

Now we state our results corresponding to different sizes of holes. The case of critical size of holes is considered by Feireisl-Namlyeyeva-Ne\v{c}asov\'a in \cite{ref14}, so we are focusing only on the case of small holes and large holes.  Note that the limits are taken up to possible extractions of subsequences. The first result corresponds to the case of small holes:
\begin{theorem}\label{thm-1}
Let $\uu_\varepsilon$ be a finite energy weak solution of the Navier-Stokes system \eqref{1-equation} in the sense of Definition \ref{def-weak} with initial datum and external force satisfying \eqref{ini-force}. Let $\widetilde  p_\e$ be the extension of $p_\e$ defined by $\widetilde  p_\varepsilon = \partial _t\widetilde P_\varepsilon$  where $\widetilde P_\varepsilon$ is the extension of $ P_{\e}$ defined in \eqref{def-F-P} and \eqref{small-F-P}. If $\alpha>3$, i.e. $\mathop {\lim }\limits_{\varepsilon \to\infty} \sigma _\varepsilon = \infty$, then
\ba\label{thm1-conv}
{\widetilde {\uu}_\varepsilon } \to {\uu} \ \mbox{weakly(*) in} \  L^{\infty}(0,T;L^{2}(\Omega)) \cap L^2(0,T;W_0^{1,2}(\Omega)),
\ea
and
\ba\label{thm1-conv-p}
\widetilde p_\varepsilon \to p\ \mbox{weakly\ in} \ W^{-1,2} (0,T;L_0^2(\Omega)).
\ea
Moreover, $(\uu,p)$ is a weak solution of the Navier-Stokes equations in homogeneous domain $\Omega$:
\begin{equation}
\begin{cases} \label{equa-u-p-small}
\partial_t \uu + \dive (\uu \otimes \uu) - \mu \Delta \uu + \nabla p = \ff,&\ \mbox{in} \;\Omega \times (0,T) ,\\
\dive \uu = 0, &\ \mbox{in} \;{\Omega} \times  (0,T),\\
\uu= 0, &\ \mbox{on} \; \partial\Omega \times (0,T),\\
\uu{|_{t=0}} = \uu^0.
\end{cases}
\end{equation}
\end{theorem}

\medskip

For the case of large holes, we consider the time-scaled Navier-Stokes system:
\begin{equation}\label{1-time-equa}
\begin{cases}
\sigma_\varepsilon^2 \partial _t \uu_\varepsilon  + \dive (\uu_\varepsilon \otimes \uu_\varepsilon)
-\mu \Delta \uu_\varepsilon  + \nabla p_\varepsilon = \ff, &\ \mbox {in} \;\Omega_\varepsilon \times (0,T),\\
\dive\uu_\varepsilon = 0, &\ \mbox{in} \; \Omega _\varepsilon \times (0,T),\\
\uu_\varepsilon = 0, &\ \mbox{on} \; \partial\Omega _\varepsilon \times (0,T),\\
\uu_\varepsilon |_{t = 0} = \uu^0_{\e}\in L^2(\Omega_\e;\mathbb{R}^3).
\end{cases}
\end{equation}

For a solution $\uu_{\e}$ to \eqref{1-time-equa},  we similarly introduce $\UU_{\e},  \Psi_{\e}, \FF_{\e}$ as in \eqref{1-new-def}.  There exists $P_\varepsilon\in C([0,T];L^2_{0}(\Omega_\varepsilon))$, such that for any $t\in (0,T)$,
\begin{equation}\label{2-equa-large-holes}
\sigma_{\e}^{2} \uu_\varepsilon(t) - \sigma_{\e}^{2} \uu^0_\e+\Phi _\varepsilon(t) - \mu \Delta \UU_\varepsilon (t)+\nabla P_\varepsilon(t) = \FF_\e(t),\quad\text{in}~W^{ - 1,2}(\Omega _\varepsilon).
\end{equation}

We have the following theorem concerning the case of large holes:
\begin{theorem}\label{thm-3}
Let $\uu_\varepsilon$ be a finite energy weak solution of the time-scaled Navier-Stokes system \eqref{1-time-equa} in the sense of Definition \ref{def-weak} with initial datum and external force satisfying \eqref{ini-force}. Let $\widetilde  p_\e$ be the extension of $p_\e$ defined by $\widetilde  p_\varepsilon = \partial _t\widetilde P_\varepsilon$ where $\widetilde P_\varepsilon$ is the extension of $ P_{\e}$ defined in \eqref{large-F-P-1} and \eqref{large-F-P-2}.  If $1<\alpha<3$, i.e. $\lim _{\varepsilon  \to 0}{\sigma _\varepsilon}=0$, then
\ba\label{thm2-u}
\sigma_{\e}^{-2} \uue \to {\uu}  \ \mbox{weakly in} \ {L^2}(0,T;L^{2}(\Omega ;{\mathbb{R}^3})) ,
\ea
and
\ba\label{thm2-p-1}
\widetilde p_\varepsilon=\widetilde p_\varepsilon^{(1)}+\sigma_\e^{\frac12}\widetilde p_\varepsilon^{(2)}+\sigma_\e^2\widetilde p_\varepsilon^{(3)},
\ea
with
\ba\label{thm2-p-2}
&\widetilde p_\varepsilon^{(1)}  \to p \ \mbox{weakly in} \ L^2 (0,T;W^{1,2}(\Omega)),\\
&\widetilde p_\varepsilon^{(2)} \ \mbox{bounded in} \ L^\frac{4}{3}(0,T;L_0^2(\Omega)),\\
&\widetilde p_\varepsilon^{(3)}  \ \mbox{bounded in} \ W^{-1,2} (0,T;L_0^2(\Omega)).
\ea
Moreover $({\uu},p)$ is a weak solution of the Darcy's law:
\begin{equation}\label{Darcy}
\begin{cases}
\mu \uu = A({\ff} - \nabla p),&\ \mbox{in} \;\Omega  \times  (0,T),\\
\dive\uu = 0, &\ \mbox{in} \;\Omega \times (0,T) ,\\
\uu\cdot \mathbf{n} = 0, &\ \mbox{on} \;\partial\Omega \times (0,T),
\end{cases}
\end{equation}
where $\mathbf{n}$ is the unit normal vector on the boundary of $\Omega$.

\end{theorem}

Here in \eqref{Darcy}, the permeability tensor $A$ is a constant positive definite matrix (see \cite{Allaire91}) given by:
\[A_{i,j}: = \mathop {\lim}\limits_{\eta \to 0} c_\eta ^{- 2}\int_{Q_\eta} \nabla  w_\eta ^i:\nabla w_\eta ^j\dd x
= \mathop {\lim }\limits_{\eta \to 0} \int_{Q_\eta} {(w_\eta ^j)}_i \dd x: = ({\overline  w}^j)_i,\]
where $w_\eta ^i$ satisfies the generalized cell problem introduced in Section \ref{sec:cell}.

\begin{remark}
The case $\alpha=1$ is considered by Mikeli\'c \cite{ref2}.  In this case the permeability tensor $A_{i,j}$ is defined by the classical cell problem where $\eta = 1$ (see  also \cite{ref7}) which is slightly different from the definition above.
\end{remark}

 Sections \ref{sec-thm1} and  \ref{sec-thm2} are devoted to the proofs of Theorem \ref{thm-1} and Theorem \ref{thm-3} respectively. In the sequel,  $C$ denotes a constant independent of $\varepsilon$, while its value may differ from line to line.

\subsection{Generalized cell problem}\label{sec:cell}
We introduce the idea of \textit{generalized cell problem} \cite{ref19, ref1,ref7} which is used to study the homogenization process. Near each single hole, after a scaling of size $\varepsilon^{-1}$ such that the controlling cube becomes the $O(1)$ size,  one can consider the following \textit{modified cell problem}:
\[\begin{cases}
-\Delta w_\eta ^i + \nabla q_\eta ^i = c_\eta ^2{e^i},&\text{in}\;Q_\eta : = {Q_0}\backslash (\eta T_{0}),\\
\dive w_\eta ^i = 0, & \mbox{in}\;Q_\eta,\\
w_\eta ^i = 0, & \mbox{on}\;\eta T_{0},\\
(w_\eta ^i,q_\eta ^i)   \ \mbox{is  $Q_0$-periodic}.
\end{cases}\]
Here $Q_{0} = \big(-\frac 12, \frac 12\big)^{d}, \ \eta : = \frac{a_\varepsilon}{\varepsilon},\ c_\eta := \frac{\varepsilon }{\sigma _\varepsilon}$, and $\{ e^i\}_{i = 1,...,d}$ is the standard Euclidean coordinate of ${\mathbb{R}^d}$. Clearly ${c_\eta } \to 0$ when $\eta  \to 0$.  When $a_\varepsilon$ is proportional to $\varepsilon$, $\eta$ becomes a positive constant independent of $\varepsilon$ and $Q_\eta$ becomes a fixed domain of type ${Q_0}\backslash T_{0}$; this is the case considered by Tartar \cite{ref7}.

For each fixed $\eta>0$, the generalized cell problem admits a unique regular solution. We now recall two lemmas concerning the estimates of the cell problems. The proofs can be found in \cite{ref1}.

\begin{lemma}
The solution $(w^i_\eta, q_\eta^i)$ of the generalized cell problem has the estimates:
\ba\label{lem-cell-1}
\|  \nabla w_{\eta}^i\| _{L^2(Q_\eta)}\leqslant C c_\eta,\quad\|  w_\eta ^i\| _{L^2(Q_\eta)} \leqslant C,\quad
\|  q_{\eta}^i\| _{L^2(Q_\eta)}\leqslant C c_\eta.
\ea
\end{lemma}
Define the scaled cell solutions
\[w_{\eta ,\varepsilon }^i( \cdot ): = w_\eta ^i(\frac{ \cdot }{\varepsilon }),\quad q_{\eta ,\varepsilon }^i( \cdot ): = q_\eta ^i(\frac{ \cdot }{\varepsilon })\]
which solves
\ba\label{cell-general}
\begin{cases}
-\varepsilon ^2\Delta w_{\eta,\varepsilon}^i + \varepsilon\nabla q_{\eta,\varepsilon }^i = c_\eta ^2{e^i},&{\text{in}}\;\varepsilon {Q_0}\backslash ({a_\varepsilon}T_{0}),\\
\dive w_{\eta,\varepsilon}^i= 0, & \mbox{in} \;{\varepsilon Q_\eta},\\
w_{\eta, \varepsilon}^i= 0, & \mbox{on} \;a_\varepsilon T_{0},\\
(w_{\eta, \varepsilon}^i,q_{\eta ,\varepsilon }^i) \   \mbox{is $\e Q_0$-periodic.}
\end{cases}
\ea
Employing the estimates of $(w_\eta ^i,q_\eta ^i)$ in \eqref{lem-cell-1} gives
\begin{lemma}
The scaled cell soluiton $(w_{\eta, \varepsilon}^i, q_{\eta, \varepsilon}^i)$ has the estimates:
\ba\label{lem-cell-2}
&\|  w_{\eta, \varepsilon}^i\| _{L^2(\Omega )} \leqslant C\| {w_\eta ^i}\| _{L^2(Q_0)} \leqslant C,\\
&\|  q_{\eta, \varepsilon}^i\| _{L^2(\Omega )} \leqslant C\| q_{\eta, \varepsilon}^i\| _{L^2(Q_0}
\leqslant C {c_\eta},\\
&\|  \nabla w_{\eta, \varepsilon}^i\| _{L^2(\Omega)}
\leqslant C \varepsilon^{-1}\|   \nabla w_\eta ^i\| _{L^2(Q_0)}
\leqslant C \varepsilon^{-1}{c_\eta} \leqslant C\sigma_\varepsilon^{-1}.
\ea
\end{lemma}
From these estimates we have
\begin{equation}\label{cell-3}
w_{\eta,\varepsilon}^i \to \overline w^i\ \mbox{weakly} \ \mbox{in} \ L^2 (\Omega),\quad
c_\eta^{-1}q_{\eta, \varepsilon}^i \to \overline q^i\ \mbox{weakly} \ \mbox{in}\ L^2(\Omega).
\end{equation}

%
%
%


\section{Proof of Theorem \ref{thm-1}}\label{sec-thm1}

This section is devoted to proving Theorem \ref{thm-1} concerning the case of small holes where $a_{\e} = \e^{\alpha}$ with $\alpha >3$ and $\sigma_{\e} = \e^{\frac{3-\alpha}{2}} \to \infty$ \eqref{ae-se}.

\subsection{Estimates of velocity}
In this case, the uniform estimates of $\uu_{\e}$ follow directly  from the energy inequality \eqref{6}. Indeed, using H\"older's inequality and Poincar\'e inequality gives
\ba\nn
&\frac{1}{2}\int_{\Omega_\varepsilon} \left| {{{\uu}_\varepsilon}(x,t)} \right|^2 \dd x
+ \int_0^t \int_{\Omega _\varepsilon} {\left| {\nabla {{\uu}_\varepsilon }(x,s)} \right|}^2 \dd x\dd s\\
&\leqslant \int_0^t \int_{{\Omega _\varepsilon }} \ff_{\e} \cdot \uu_\varepsilon  \dd x\dd s
+ \frac{1}{2}\int_{{\Omega _\varepsilon }} \left| {{\uu}^0_{\e}(x)} \right|^2 \dd x\\
& \leqslant C \sup_{0<\e\leqslant 1}\| \ff_{\e}\|  _{L^{2}(0,T; {L^2}({\Omega _\varepsilon }))}^{2} + \frac{1}{2}\| {{{\nabla \uu}_\varepsilon }}\| _{L^{2}(0,t; {L^2}({\Omega _\varepsilon }))}^{2} + \frac{1}{2}\sup_{0<\e\leqslant1} \|\uu_{\e}^{0}\|_{L^{2}(\Omega_{\e})}^{2}.
\ea
Together with the assumption on the initial datum and the external force in \eqref{ini-force}, we deduce
\begin{equation}\label{est-u-0}
\|  \uu_\varepsilon \| _{{L^\infty}(0,T;{L^2}({\Omega _\varepsilon}))} \leqslant C, \quad\| {\nabla {{\uu}_\varepsilon }}\| _{{L^2}(0,T;{L^2}({\Omega _\varepsilon }))} \leqslant C.
\end{equation}
Since ${\uu}_\varepsilon \in L^{2}(0,T; W^{1,2}_{0}(\Omega_{\e}))$ has zero trace on the boundary, its zero extension
$\widetilde  \uu_{\e} \in L^{\infty}(0,T;  \\ L^{2}(\Omega)) \cap L^{2}(0,T; W^{1,2}_{0}(\Omega))$ has the estimates:
\begin{equation}\label{est-u}
\|  \widetilde \uu_\varepsilon \| _{L^\infty(0,T;L^2(\Omega))} \leqslant C,
\quad\| \widetilde \uu_\varepsilon \| _{L^2 (0,T;W_0^{1,2}(\Omega))} \leqslant C.
\end{equation}
Thus, up to a subsequence, there holds the convergence
\ba\label{thm1-conv-0}
{\widetilde {\uu}_\varepsilon } \to {\uu} \ \mbox{weakly(*) in} \  L^{\infty}(0,T;L^{2}(\Omega)) \cap L^2(0,T;W_0^{1,2}(\Omega)),
\ea
which is exactly \eqref{thm1-conv} in Theorem \ref{thm-1}.  Moreover, by the definition of $\UU_\e$ in \eqref{1-new-def} we have
\ba\label{est-small-U}
\| \widetilde \UU_\varepsilon  \|_{W^{1,\infty}((0,T); L^2(\Omega))}\leqslant C,
\quad \|\widetilde \UU_\varepsilon  \|_{W^{1,2}((0,T); W^{1,2}_{0}(\Omega))}\leqslant C, \quad \|  \widetilde \UU_\varepsilon  \|_{C([0,T]; W^{1,2}_{0}(\Omega))}\leqslant C.
\ea

\subsection{Extension of pressure}\label{sec:extP-small}
The extension of the pressure is given by using the dual formula and employing the so-called {\em restriction operator} due to Allaire \cite{ref3, ref4} for general sizes of holes, and due to Tartar \cite{ref7} for the case where the size of the holes is proportional to their mutual distance. A restriction operator $R_{\e} $ is a linear operator $R_{\e} : W^{1,2}_{0}(\Omega;\RR^{d}) \to W^{1,2}_{0}(\Omega_{\e};\RR^{d}) $ such that:
\ba\label{pt-res}
&\uu \in W_0^{1,2}(\Omega_\e;\RR^d) \Longrightarrow R_\e (\widetilde  \uu)=\uu \ \mbox{in}\ \Omega_\e,\ \mbox{where} \ \widetilde
\uu:=\begin{cases}\uu \ &\mbox{in}\ \Omega_\e,\\ 0  \ &\mbox{on}\ \Omega\setminus \Omega_\e, \end{cases}\\
&\uu \in W_0^{1,2}(\Omega;\RR^d),\  \dive\uu =0 \ \mbox{in} \ \Omega \Longrightarrow \dive R_\e (\uu) =0 \ \mbox{in} \ \Omega_\e,\\
&\uu \in W_0^{1,2}(\Omega;\RR^d)\Longrightarrow \|\nabla R_\e(\uu)\|_{L^2(\Omega_\e)} \leqslant C \, \big( \|\nabla \uu\|_{L^2(\Omega)} + (1+\sigma_{\e}^{-1}) \|\uu\|_{L^2(\Omega)}\big).
\ea
For each $\varphi\in L^{q}(0,T; W^{1,2}_0(\Omega))$ with $1<q<\infty$, the restriction $R_{\e}(\varphi)$ is taken only on spatial variable:
\[R_\e(\varphi)(\cdot,t)=R_\e(\varphi(\cdot,t))(\cdot) \quad \mbox{for each $t\in (0,T)$}. \]
Clearly $R_{\e}$ maps  $L^{q}(0,T; W^{1,2}_0(\Omega))$ onto $L^{q}(0,T; W^{1,2}_0(\Omega_{\e}))$ with the estimate:
\ba\label{pt-res-2}
\|\nabla R_\e(\varphi)\|_{L^{q}(0,T; L^{2}(\Omega_\e))} \leqslant C \, \big( \|\nabla \varphi\|_{L^{q}(0,T; L^{2}(\Omega))} + (1+\sigma_{\e}^{-1}) \|\varphi\|_{L^{q}(0,T; L^{2}(\Omega))}\big).
\ea
\begin{lemma}\label{lemma-large}
Let $1<q<\infty$. Assume $H \in L^{q}(0,T; W^{-1,2}(\Omega;\RR^{3}))$ satisfying
\ba
\l H, \varphi \r_{\Omega\times(0,T)}=0,\ \forall  \varphi\in C_{c}^{\infty}(\Omega \times (0,T)  ;\RR^{3}), \ \dive \varphi=0.
\nn\ea
 Then there exists a scalar function $P\in L^{q}(0,T; L_0^2(\Omega))$ such that:
\ba\label{H-nablaP}
H = \nabla P, \ \mbox{with} \  \| P\| _{L^{q}(0,T;L_0^2(\Omega))} \leqslant C\| H\| _{L^{q}(0,T;W^{-1,2}(\Omega))}.
\ea
\end{lemma}
\begin{proof}
For each $\psi \in C_c^\infty(\Omega;\RR^{3})$ with $\dive \psi\equiv 0$ and $ \phi\in C_c^\infty(0,T)$, there holds
\ba\nonumber
0 = \l H, \phi(t)\psi(x) \r_{\Omega\times(0,T)} =  \int_0^T \l H(\cdot,t), \psi(\cdot) \r_{\Omega} \phi(t)\dd t.
\ea
Thus we have for a.a. $t\in (0,T)$ that
\ba\nonumber
\l H(\cdot,t), \psi(\cdot) \r_{\Omega}=0, \quad \forall \psi \in C_c^\infty(\Omega;\RR^{3}), \ \dive \psi\equiv 0.
\ea
Therefore,  for a.a. $t\in(0,T)$, there exists $P(\cdot,t)\in L_0^2(\Omega)$ (see \cite{ref8}), such that
\ba\nonumber
H (\cdot,t)=\nabla P(\cdot, t),\ \mbox{with}\  \| P(\cdot, t)\| _{L_0^2(\Omega)} \leqslant\| H (\cdot,t)\| _{W^{-1,2}(\Omega)}.
\ea
This implies immediately \eqref{H-nablaP}.

\end{proof}

Now we define a functional $\widetilde H_\e$ in $\mathcal{D}'(\Omega\times(0,T))$ by the following dual formulation:
\ba\nn
\l\widetilde H_\e,\varphi\r_{\Omega\times(0,T)}=\l\nabla P_\e,R(\varphi)\r_{\Omega_\e\times(0,T)},\ \forall \varphi \in C_c^\infty(\Omega\times(0,T);\RR^{3}),
\ea
where $P_{\e}\in C([0,T],L_{0}^{2}(\Omega_{\e}))$ is given in \eqref{2-equa}. Then for any $\varphi  \in C_c^\infty  (\Omega\times(0,T); \mathbb{R}^3)$,
\ba\label{def-F-P}
\l \widetilde H_\e,\varphi \r_{\Omega\times(0,T)}&
=\l\nabla {P_\varepsilon}(t), R_\varepsilon (\varphi) \r_{\Omega_\e\times(0,T)}\\
&=\l\FF_\e(t)-\uu_\varepsilon(t)+\uu^0_\e + \mu \Delta \UU_\varepsilon(t)
-{\Phi_\varepsilon}(t),{R_\varepsilon}(\varphi) \r_{\Omega_\e\times(0,T)}.
\ea

Using the uniform estimates in \eqref{est-u-0}--\eqref{est-small-U} implies
\ba\label{est-P-small-1}
\left|  \l \Delta \UU_\varepsilon(t), R_\varepsilon (\varphi) \r_{\Omega_\varepsilon\times(0,T)} \right|
&\leqslant\| \nabla \UU_\varepsilon(t)\| _{L^2 L^2}
\|  \nabla R_\varepsilon (\varphi)\| _{L^2 L^2} \leqslant C\|   \varphi \| _{L^2 W^{1,2}_0},\\
\left|  \l\Phi_\varepsilon(t), R_\varepsilon(\varphi)\r_{\Omega_\varepsilon\times(0,T)} \right|
&= \left|  \l \Psi_\varepsilon(t), \nabla R_\varepsilon(\varphi)\r_{\Omega_\varepsilon\times(0,T)} \right| \leqslant\|   \Psi_{\e}\| _{L^\infty L^2}\|  \nabla R_{\e} (\varphi )\| _{L^2 L^{2}}\\
&\leqslant C\|\uu_{\e}\otimes \uu_{\e}\|_{L^{1}L^{2}}\| \varphi\| _{L^2 W^{1,2}_0}\leqslant  C\|  \varphi\| _{L^2 W^{1,2}_0},\\
\left| \l {\uu}_\e, R_\varepsilon (\varphi) \r_{\Omega_\varepsilon\times(0,T)} \right|
&\leqslant C\| \uu_\e\| _{L^\infty(0,T; L^2(\Omega _\varepsilon))} \|    R_\varepsilon (\varphi) \| _{L^2 L^2}\leqslant C\| \varphi \| _{L^2 W^{1,2}_0},\\
\left| \l {\uu}^0_\e, R_\varepsilon (\varphi) \r_{\Omega _\varepsilon\times(0,T)} \right|
&\leqslant C \| \uu^0_\e\| _{L^2} \| R_\varepsilon (\varphi)\| _{L^2 L^2} \leqslant C\|  \varphi\| _{L^2 W^{1,2}_0},\\
\left| \l \FF_\e, R_\varepsilon (\varphi) \r_{\Omega _\varepsilon\times(0,T)} \right|
&\leqslant C \| \FF_\e\| _{L^2 L^2}\| R_\varepsilon (\varphi) \| _{L^2 L^2}\leqslant C\|   \varphi\| _{L^2 W^{1,2}_0}.
\ea
Estimates in \eqref{est-P-small-1} imply
\ba\label{est-small-F}
\l \widetilde H_\e,\varphi \r_{\Omega\times(0,T)}
&\leqslant C \| \varphi \| _{L^2 W^{1,2}_0}.
\ea
Thus $\widetilde H_\e$ is bounded in $L^{2}(0,T; W^{-1,2}(\Omega;\RR^{3}))$. Moreover, the second property of the restriction operator in \eqref{pt-res} implies that
\ba
\l \widetilde H_\e, \varphi \r_{\Omega\times(0,T)}=0,\ \forall  \varphi\in C_{c}^{\infty}(\Omega \times (0,T)  ;\RR^{3}), \ \dive \varphi=0.
\nn\ea
Thus, by \eqref{est-small-F}, we can apply Lemma \ref{lemma-large} and deduce that there exists $\widetilde P_\e \in L^2(0,T; L^2_0(\Omega)) $ such that
\ba\label{small-F-P}
\widetilde H_\e=\nabla \widetilde P_\e,
\ea
and
\ba\label{est-small-P}
\|  \widetilde P_\e \|_{L^2(0,T; L^2_0(\Omega))} \leqslant C \| \widetilde H_\e \|_{L^2(0,T; W^{-1,2}(\Omega))} \leqslant C.
\ea
\subsection{Momentum equations in the homogeneous domain}

In the case of small holes, we find that the extension $\widetilde \uu_\varepsilon$ satisfies the Navier-Stokes equations in $\Omega$ up to a small remainder:
\begin{proposition}\label{moment-equa}
Under the assumptions in Theorem \ref{thm-1}, the extension $\widetilde \uu_\varepsilon$ satisfies the following equations in the sense of distribution:
\[\partial _t \widetilde \uu_\varepsilon + \dive \left(\widetilde \uu_\varepsilon  \otimes \widetilde \uu_\varepsilon \right)- \mu \Delta \widetilde \uu_\varepsilon + \nabla \hat  p_\varepsilon =  \widetilde \ff_\e +\mathbf{G}_\varepsilon, \ \dive \widetilde \uu_{\e} = 0, \]
where $\GG_\varepsilon \in \mathcal{D}' (\Omega \times (0,T))$ satisfying
\begin{equation}\label{est-F-3D}
\left| \l \GG_\varepsilon, \varphi \r \right| \leqslant C\varepsilon ^\sigma \big(\|  \partial _t \varphi \| _{L^{\frac 43} L^{2}}
+\| \nabla \varphi\| _{L^{4} L^{r_1}} \big), \ \forall \, \varphi\in C_{c}^{\infty}(\Omega\times (0,T);\mathbb{R}^3), \ \dive \varphi =0.
\end{equation}
Here $\sigma:=((3 - q)\alpha  - 3)/q >0$ for some $q>2$ close to $2$, and $ 2<r_1<3$ given in \eqref{r1}.
\end{proposition}
\begin{proof}
Let $\varphi  \in C_c^\infty (\Omega \times (0,T);{\mathbb{R}^3})$ with $\dive\varphi = 0$. To extend the Navier-Stokes equations  from $\Omega$ to $\Omega_\varepsilon$, an idea is to find a family of functions ${\{ {g_\varepsilon}\} _{\varepsilon>0}}$ vanishing on the holes and converges to $1$ in some Sobolev space $W^{1,q}(\Omega)$ and decompose $\varphi$ as
\be\label{dec-g}
\varphi  = {g_\varepsilon }\varphi + (1-{g_\varepsilon})\varphi.
\ee
Then $g_\varepsilon\varphi$ can be treated as a test function for the momentum equations in $\Omega_\varepsilon$. While for the terms related the other part $(1-g_\varepsilon)\varphi$, we show that they are small and converge to zero. However, such a decomposition destroyed the divergence free property of $\varphi$: $\dive (g_{\e} \varphi) \neq 0$. To overcome this trouble, we introduce the following Bogovskii type operator in perforated domain $\Omega_{\e}$ (see Proposition 2.2 in \cite{ref9} and Theorem 2.3 in \cite{ref15}):
\begin{lemma}\label{lem-Bog}
Let $\Omega_\e$ defined as in \eqref{1-hole} and \eqref{1-domain} with $\alpha\geq 1$. Then for any $1<q<\infty$, there exists a linear operator $\mathcal{B}_\e \,: \, L^q_{0}(\Omega_\e)\to W_{0}^{1,q}(\Omega_\e;\RR^3)$ such that for any $f\in L^q_{0}(\Omega_\e)$, there holds
\be\label{pro-div1}
\dive \mathcal{B}_\e(f)=f\  \mbox{in} \ \Omega_\e, \quad \|\mathcal{B}_\e(f)\|_{W_{0}^{1,q}(\Omega_\e;\RR^3)}\leq C\big(1+\e^{\frac{(3-q)\aaa-3}{q}}\big)\|f\|_{L^q(\Omega_\e)}
\ee
for some constant $C$ independent of $\e$.

  For any $r >3/2$, the linear operator $\mathcal B_\e$ can be extended as a linear operator from $\{\dive \vg : \vg\in L^r(\Omega_\e;\RR^3),
  \vg \cdot {\bf n}=0 \mbox{ on } \p \Omega_\e\}$ to $L^r(\Omega_\e;\RR^3)$ satisfying
\be\label{pro-div2}
 \|\mathcal{B}_\e(\dive {\bf g})\|_{L^{r}(\Omega_\e;\RR^3)}\leq C \|{\bf g}\|_{L^r(\Omega_\e;\RR^3)},
\ee
for some constant $C$ independent of $\e$.
\end{lemma}

By the description of the holes in \eqref{1-hole},  there exists cut-off functions $\{ g_\varepsilon \}_{\varepsilon>0} \subset C^\infty(\RR^{3})$ such that $0\leqslant g_{\e} \leqslant 1$ and
\begin{equation}\nn
  g_\varepsilon= 0 \  \mbox{on}  \  \bigcup_{k \in K_\varepsilon} B(\varepsilon {x_k},\delta_1 \varepsilon^\alpha),\quad g_\varepsilon = 1  \  \mbox{on}  \  (\bigcup_{k \in K_\varepsilon} B(\varepsilon {x_k},\delta_2 \varepsilon^\alpha))^c,\quad |\nabla g_{\e}| \leqslant C \e^{-\alpha}.
\end{equation}
Then for each $1\leqslant q \leqslant \infty$ there holds
\begin{equation}\label{est-g}
\|  g_\varepsilon- 1\| _{L^q(\RR^{3})} \leqslant C \varepsilon^{\frac{3\alpha-3}{q}},\quad
\|  \nabla g_\varepsilon\| _{L^q(\RR^{3})} \leqslant C\varepsilon ^{\frac{3\alpha-3}{q}-\alpha}.
\end{equation}
Now we estimate
\ba
I^\varepsilon &:= \int_0^T \int_\Omega \widetilde \uu_\varepsilon \partial _t \varphi
+\widetilde \uu_\varepsilon  \otimes \widetilde \uu_\varepsilon :\nabla \varphi
-\nabla \widetilde \uu_\varepsilon :\nabla \varphi + \widetilde \ff_\e \varphi \, \dd x \dd t.
\nonumber\ea
Using the decomposition \eqref{dec-g} we write
\ba
I^\varepsilon &=\int_0^T \int_{\Omega_\varepsilon}\uu_\varepsilon \partial_t (g_\varepsilon \varphi)
+ \uu_\varepsilon \otimes  \uu_\varepsilon:\nabla (g_\varepsilon\varphi)
-\nabla  \uu_\varepsilon:\nabla (g_\varepsilon\varphi ) +  \ff_\e (g_\varepsilon \varphi )\dd x \dd t  +\sum_{j=1}^{4} {I_j}
\nonumber\ea
with
\ba\nn
I_1 & = \int_0^T \int_\Omega \widetilde \uu_\varepsilon(1-g_\varepsilon) \partial _t\varphi \dd x \dd t, \\
I_2 & = \int_0^T \int_\Omega \widetilde \uu_\varepsilon  \otimes \widetilde \uu_\varepsilon :(1-g_\varepsilon)\nabla \varphi
-\widetilde \uu_\varepsilon \otimes \widetilde \uu_\varepsilon :\nabla g_\varepsilon \otimes \varphi \dd x \dd t ,\\
I_3 &= \int_0^T \int_\Omega \nabla \widetilde \uu_\varepsilon :(1-g_\varepsilon)\nabla \varphi
- \nabla \widetilde \uu_\varepsilon:\nabla g_\varepsilon \otimes \varphi \dd x \dd t ,\\
I_4 &= \int_0^T \int_\Omega \widetilde \ff_\e (g_\varepsilon  - 1)\varphi \dd x\dd t.
\ea

Observe that
$$
\int_{\Omega_{\e}} \varphi  \cdot \nabla g_{\e}\dd x = \int_{\Omega_{\e}} \dive(\varphi g_{\e}) \dd x  = 0.
$$
Thus, we can apply Lemma \ref{lem-Bog} and introduce
\ba\nn
\varphi_{1} := \varphi - \varphi_{2},\quad  \varphi_{2} := \calB_{\e} (\dive(\varphi g_{\e})) = \calB_{\e} (\varphi  \cdot \nabla g_{\e}).
\ea
Then $\varphi_{1} \in C_{c}^{\infty}( (0,T); W^{1,2}_{0}(\Omega_{\e}))$ satisfying $\dive \varphi_{1} = 0$. Using the weak formulation \eqref{5} gives
\ba
I^\varepsilon & = \int_0^T \int_{\Omega_\varepsilon}\uu_\varepsilon \partial_t \varphi_{1}
+ \uu_\varepsilon \otimes  \uu_\varepsilon:\nabla \varphi_{1}
-\nabla  \uu_\varepsilon:\nabla \varphi_{1} +  \ff_\e \cdot  \varphi_{1} \dd x \dd t   + \sum_{ j=1}^{4} {I_j} +   \sum_{ j=5}^{8} {I_j} \\
&= \sum_{ j=1}^{4} {I_j} +   \sum_{ j=5}^{8} {I_j} ,
\nonumber\ea
where
\ba
I_5 &=\int_0^T \int_{\Omega_\varepsilon}\uu_\varepsilon \partial_t \varphi_{2}  \dd x \dd t,\quad &&I_6 = \int_0^T  \int_{\Omega_\varepsilon} \uu_\varepsilon \otimes  \uu_\varepsilon : \nabla \varphi_{2}   \dd x \dd t , \\
I_7 &= \int_0^T \int_{\Omega_\varepsilon} -\nabla  \uu_\varepsilon:\nabla \varphi_{2}  \dd x \dd t ,\quad &&I_8 = \int_0^T \int_{\Omega_{\e}}  \ff_\e \cdot  \varphi_{2}\dd x\dd t.
\nonumber \ea

Now we estimate $I_{j}$ term by term.  Since $\alpha>3$, there exists $q\in(2, 3)$ close to $2$ such that
\be\nn
\sigma:=\frac{(3 - q)\alpha  - 3}{q} >0.
\ee
Let $q^*$ be the Sobolev conjugate component to $q$ with $\frac{1}{q^{*}} = \frac{1}{q} - \frac{1}{3}$. Clearly $6<q^{*}<\infty$.

\medskip

By the uniform estimates \eqref{est-u} and \eqref{est-g}, together with interpolation and Sobolev embedding, we obtain
\[|I_1| \leqslant\| \widetilde \uu_\varepsilon \| _{L^4 L^3}\| 1- g_\varepsilon \| _{L^6}\| \partial_t\varphi\| _{L^{\frac 43} L^2} \leqslant   \| \widetilde \uu_\varepsilon \| _{L^\infty L^2 \cap L^{2} L^{6}} \| 1- g_\varepsilon \| _{L^6}\| \partial_t\varphi\| _{L^{\frac 43} L^2}   \leqslant C \varepsilon ^\sigma\| \partial _t\varphi \| _{L^{\frac 43} L^2} .\]

Again by interpolation and using Sobolev embedding, we have
\ba
|I_2| \leqslant \| \widetilde \uu_\varepsilon  \otimes \widetilde \uu_\varepsilon \| _{L^{\frac{4}{3}} L^2}\big(
\|  g_\varepsilon - 1\| _{L^{q^{*}}}\|  \nabla \varphi \| _{L^4L^{r_{1}}} + \| \nabla g_\varepsilon \| _{L^q}\| \varphi \| _{L^{4} L^{r_2}}\big) \leqslant C \varepsilon ^\sigma\| \nabla \varphi \| _{L^{4} L^{r_1}} ,
\nonumber\ea
where
\ba\label{r1}
\frac{1}{r_{1}} = \frac{1}{2}  - \frac{1}{q_{*}} = \frac{5}{6} - \frac{1}{q}>\frac{1}{3}, \quad \frac{1}{r_{2}} = \frac{1}{2} - \frac{1}{q}  = \frac{1}{r_{1}^{*}}.
\ea

Similarly,
\ba
|I_3| \leqslant\| \nabla \widetilde \uu_\varepsilon \| _{L^2 L^2} \big(\|  g_\varepsilon - 1\| _{L^{q^*}}
\|  \nabla \varphi\| _{L^2 L^{r_1}} +\| \nabla g_\varepsilon\| _{L^q}\|  \varphi\| _{L^2 L^{r_2}}\big)\leqslant C\varepsilon ^\sigma\| \nabla \varphi\| _{L^2 L^{r_1}}.
\nonumber\ea

For $I_{4}$,
\[|I_4| \leqslant\| \ff_\e\| _{L^2 L^2}\| g_\varepsilon - 1\| _{L^{q^*}}\|  \varphi \| _{L^2 L^{r_1}}
\leqslant C \varepsilon ^\sigma\| \varphi \| _{L^2 L^{r_1}}.\]

\medskip

Next we estimate $I_{j}, j=5,6,7,8$ for which the estimates of the Bogovskii operator $\calB_{\e}$ in \eqref{pro-div1} and \eqref{pro-div2} will be repeatedly used. Since the Bogovskii operator $\calB_{\e}$ only applies on spatial variable, then $$\p_{t} \varphi_{2} = \p_{t} \calB_{\e}(\varphi  \cdot \nabla  g_{\e}) = \calB_{\e}(\p_{t}\varphi \cdot \nabla  g_{\e}) =  \calB_{\e} \big(\dive((\p_{t} \varphi)   g_{\e})\big) = \calB_{\e} \big(\dive((\p_{t} \varphi)   (g_{\e}-1))\big).$$ Thus, taking $r_{5}>3/2$ and using \eqref{pro-div2} gives
\ba
|I_5|&  \leqslant\| \uu_\varepsilon \| _{L^4 L^3} \|\p_{t} \varphi_{2} \|_{L^{\frac{4}{3}} L^{r_{5}}}  \leqslant C \| \p_{t}\varphi ( g_\varepsilon -1 )\| _{L^{\frac 43} L^{r_{5}}}  \\
& \leqslant   C\| g_\varepsilon - 1\| _{L^{q^*}}\|   \p_{t} \varphi \| _{L^{\frac 43} L^{r_{6}}}  \leqslant C \varepsilon^\sigma\|  \p_{t} \varphi \| _{L^{\frac 43} L^{r_{6}}},
\nn\ea
where
$$
\frac{1}{r_{6}} = \frac{1}{r_{5}} - \frac{1}{q^{*}} =  \frac{1}{r_{5}} - \frac{1}{q} + \frac{1}{3}.
$$
Since $q>2$, we can choose $r_{5}>3/2$ close to $3/2$ such that
$$
  \frac{1}{r_{5}} - \frac{1}{q} = \left(\frac{1}{r_{5}} - \frac{2}{3} \right) + \left( \frac{1}{2} - \frac{1}{q} \right)+ \frac{1}{6} = \frac{1}{6}.
$$
With such a choice of $r_{5}$ we have $ r_{6} = 2$ and
\ba
|I_5|  \leqslant C \varepsilon^\sigma\|  \p_{t} \varphi \| _{L^{\frac 43} L^{2}}.
\nn\ea

Using \eqref{pro-div1}, by similar argument one has
\ba
|I_6| + |I_{7}| + |I_{8}| \leqslant C \varepsilon ^\sigma\| \nabla \varphi \| _{L^{4} L^{r_1}}.
\nonumber\ea

Summing up the above estimates for $I_{j}$ gives
$$
 |I^\varepsilon| \leqslant C\varepsilon ^\sigma \big(\|  \partial _t \varphi \| _{L^{\frac 43} L^{2}}
+\| \nabla \varphi\| _{L^{4} L^{r_1}} \big).
$$
This implies our desired result in \eqref{est-F-3D}.

\end{proof}

\subsection{Convergence of the nonlinear convective term}
Here we will show  $\widetilde \uu_{\e}$ has certain compactness such that the convergence of the nonlinear convective term $\widetilde \uu_{\e}\otimes \widetilde \uu_{\e}$ can be obtained.  A key observation is that some uniform estimates related to time derivative can be deduced from Proposition \ref{moment-equa}. Indeed, By Proposition \ref{moment-equa}, for any $\varphi  \in C_c^\infty (  \Omega \times (0,T) ;{\mathbb{R}^3})$ with $\dive\varphi=0$, we have
\ba\nn
\left| \l \partial _t \widetilde \uu_\varepsilon, \varphi \r \right| &
\leqslant \int_0^T \int_\Omega  \left| \widetilde \uu_\varepsilon
\otimes \widetilde \uu_\varepsilon :\nabla \varphi \right|  + \left| \nabla \widetilde \uu_\varepsilon:\nabla \varphi   \right| + \left|  \ff_\e \cdot \varphi \right| \dd x\dd t  + \left|  \l \GG_\varepsilon,\varphi\r  \right|\\
&\leqslant C\| \nabla \varphi\| _{L^4 L^2} + C\varepsilon^\sigma\big(\|  \partial _t \varphi \| _{L^{\frac 43} L^{2}}
+\| \nabla \varphi\| _{L^{4} L^{r_1}} \big) \\
& \leqslant C\| \nabla \varphi\| _{L^{4} L^{r_1}} + C \varepsilon^\sigma\| \partial _t \varphi \| _{L^{\frac 43} L^{2}}.
\ea
Recall that $\sigma>0$ and $r_{1} \in (2,3)$. Thus we have the following decomposition
\[\widetilde \uu_\varepsilon = \widetilde \uu_\varepsilon ^{(1)} + \varepsilon ^\sigma \widetilde \uu_\varepsilon ^{(2)},\]
where $\partial _t\widetilde \uu_\varepsilon ^{(1)}$ is uniformly bounded in $L^{\frac{4}{3}}(0,T; V^{-1,r_{1}'}(\Omega)) $ and $ \uue^{(2)}$ is uniformly bounded in $L^{4}(0,T; L^{2}(\Omega))$. Since $\uue$ is uniformly bounded in $L^{\infty}(0,T; L^{2}(\Omega))\cap L^{2}(0,T;L^{6}(\Omega))$, then $\uue^{(1)} = \uue -  \e^\sigma\uue^{(2)}$ is uniformly bounded in $L^{4}(0,T; L^{2}(\Omega))$. Thus, up to a subsequence,
$$
\uue^{(1)} \to \uu \ \mbox{weakly in} \  L^{4}(0,T; L^{2}(\Omega)).
$$

Recall $\uue$ is uniformly bounded in $L^{2}(0,T; W^{1,2}_{0}(\Omega))$. Then apply Aubin-Lions type argument (see Lemma 5.1 in \cite{ref6.1}) gives
\[\widetilde \uu_\varepsilon ^{(1)} \otimes \widetilde \uu_\varepsilon  \to \uu \otimes \uu \ \mbox{in} \ \mathcal{D}'( \Omega \times (0,T) ).\]
Clearly
$$
\e^{\sigma} \uue^{(2)} \otimes \uue \to 0 \ \mbox{strongly in} \ L^{\frac{4}{3}}(0,T; L^{\frac{3}{2}}(\Omega)) \cap L^{4}(0,T; L^{1}(\Omega)).
$$
Thus,
\ba\label{convective-conv-0}
\uue\otimes \uue \to \uu\otimes \uu \ \mbox{in} \ \mathcal{D}'( \Omega \times (0,T) ).
\ea
Together with the uniform estimates of $\uue$ in \eqref{est-u}, we deduce from \eqref{convective-conv-0} that
\ba\label{convective-conv}
\uue\otimes \uue \to \uu\otimes \uu \ \mbox{weakly in} \  L^{\frac{4}{3}}(0,T; L^{2}(\Omega) ).
\ea

\subsection{Passing to the limit}
Now we are ready to pass $\e \to 0$ and prove the following result in $\UU_{\e}$ from which we can deduce our desired limit equations in $\uu_{\e}$.
\begin{proposition}\label{thm-UP-crit}
Let $(\UU_\varepsilon, P_\varepsilon)$ be the solution of the equation \eqref{2-equa} and $\widetilde {P}_\varepsilon$ be the extension in $\Omega$ defined by \eqref{small-F-P}. If $\alpha>3$, i.e. $\lim _{\varepsilon  \to 0}\sigma _\varepsilon=\infty$,  then
\ba\label{U-P-conv}
&\widetilde \UU_\varepsilon  \to \UU \ \mbox{weakly({*}) \ in}\ W^{1,\infty}((0,T); L^2(\Omega))\cap W^{1,2}((0,T); W^{1,2}_{0}(\Omega)),\\
&\widetilde P_\varepsilon  \to P\ \mbox{weakly\ in}\ L^2(0,T; L_0^2(\Omega)).
\ea
Moreover, $(\UU,P)$ solves
\begin{equation}\label{equa-U-P-small}
\uu - \uu^0 + \dive \Psi - \mu \Delta \UU+\nabla P= \FF,\ \dive \UU=0 \  \mbox{in}\ \mathcal{D}'(\Omega\times(0,T)),
\end{equation}
with
\ba\label{U-Psi-F}
\UU(\cdot ,t) = \int_0^t \uu (\cdot,s)\dd s, \quad \Psi (\cdot ,t): = \int_0^t \uu(\cdot,s)   \otimes \uu(\cdot,s)\dd s, \quad \FF(\cdot ,t) = \int_0^t \ff (\cdot,s)\dd s,
\ea
where $\uu$ is the limit of $\uue$ given in \eqref{thm1-conv-0},  $\uu^{0}$ and $\ff$ are  given in \eqref{ini-force}.
\end{proposition}

\begin{proof}
The weak convergences in \eqref{U-P-conv} follow immediately from the uniform estimates \eqref{est-small-U} and \eqref{est-small-P}. The divergence free condition $\dive \UU = 0$ follows from $\dive \widetilde \UU_{\e}=0$. Moreover, the weak convergence of $\uue$ in \eqref{thm1-conv-0} implies for each $\varphi \in C_{c}^{\infty}(\Omega\times (0,T);\RR^{3})$ that
\ba\label{U-u-small}
\int_{0}^{T} \int_{\Omega} \widetilde \UU_{\e} \cdot \varphi(t,x) \dd x \dd t & = \int_{0}^{T} \int_{\Omega} \int_{0}^{t} \uue(s,x) \dd s \cdot \varphi(t,x) \dd x \dd t \\
 & =  \int_{0}^{T}  \int_{0}^{t} \int_{\Omega} \uue(s,x)  \cdot \varphi(t,x) \dd x \dd s \dd t \\
& \to \int_{0}^{T}  \int_{0}^{t} \int_{\Omega} \uu(s,x)  \cdot \varphi(t,x) \dd x \dd s \dd t \\
& = \int_{0}^{T} \int_{\Omega}  \int_{0}^{t} \uu(s,x) \dd s \cdot \varphi(t,x) \dd x  \dd t.
 \ea
Thus, the uniqueness of weak limits implies
$$
 \UU(\cdot ,t) = \int_0^t \uu (\cdot,s)\dd s,
$$
which is the first equality in \eqref{U-Psi-F}.

\medskip

Given any scalar test function $\phi \in C_{c}^{\infty}(\Omega \times (0,T) )$, we choose $w ^i_{\eta,\varepsilon}\phi$ as a test function to \eqref{2-equa} to obtain the weak formulation:
\ba\label{weak-equa-small}
&\int_0^T \int_\Omega  \widetilde \uu_\varepsilon  \cdot (w_{\eta, \varepsilon}^i\phi)
-\widetilde \Psi _\varepsilon:\nabla (w_{\eta, \varepsilon}^i\phi)
+\mu \nabla \UU_\varepsilon :\nabla (w_{\eta ,\varepsilon }^i\phi)
-\widetilde  P_\varepsilon\dive(w_{\eta, \varepsilon}^i\phi )\dd x\dd t \\
&=\int_0^T \int_\Omega  \FF_\e \cdot (w_{\eta,\varepsilon}^i\phi )\dd x\dd t+ \int_{0}^{T}\int_\Omega \widetilde \uu_\e^0  \cdot (w_{\eta,\varepsilon}^i\phi )  \,\dd x\dd t.
\ea

By \eqref{lem-cell-2}, we know that $\|  \nabla w_{\eta, \varepsilon}^i\| _{L^2(\Omega)} \leqslant C\sigma_\varepsilon ^{-1} = C \e^{\frac{\alpha-3}{2}} \to 0$. Then  $w_{\eta, \varepsilon}^i \to {\overline w ^i}$ strongly in $L^2(\Omega)$ by Sobolev compact embedding. This implies
\ba\label{conv-small-1}
&\int_0^T \int_\Omega \widetilde\uu_\e \cdot w_{\eta, \varepsilon}^i \phi \dd x\dd t \to
\int_0^T \int_\Omega  \uu \cdot {\overline w}^i \phi \dd x\dd t,\\
&\int_{0}^{T}\int_\Omega \widetilde \uu_\e^0  \cdot (w_{\eta,\varepsilon}^i\phi )  \,\dd x\dd t \to
\int_{0}^{T} \int_\Omega  \uu^0 \cdot {\overline w}^i \phi \dd x\dd t.
\ea

For the nonlinear convective term one has
\[\int_0^T \int_\Omega  \widetilde \Psi_\e :\nabla (w_{\eta, \varepsilon }^i\phi )\dd x\dd t
= \int_0^T \int_\Omega  \widetilde \Psi_\e:\nabla w_{\eta, \varepsilon}^i\phi
+ \widetilde \Psi_\e: \nabla \phi  \otimes w_{\eta, \varepsilon}^i  \dd x\dd t,\]
where the first term on the right-hand side satisfies
\[\left| \int_0^T \int_\Omega  \widetilde \Psi_\e :\nabla w_{\eta ,\varepsilon }^i\phi \dd x\dd t \right| \leqslant  C\| \widetilde \uu_\varepsilon \| _{L^2 L^6}^2
\|  \nabla w_{\eta, \varepsilon}^i\| _{L^2}\| \phi \| _{L^\infty L^6}
\leqslant C \varepsilon ^{\frac{\alpha-3}{2}} \to 0.\]
For the second term, by the convergence of the convective term in \eqref{convective-conv} and the strong convergence of ${w_{\eta ,\varepsilon }^i}$, we have
\ba\nonumber
\int_0^T \int_\Omega  \widetilde \Psi_\e :\nabla \phi  \otimes w_{\eta ,\varepsilon }^i\dd x\dd t \to \int_0^T \int_\Omega   \Psi :\nabla \phi \otimes  \overline w ^i\dd x\dd t.
\ea
So we obtain
\ba\label{conv-small-3}
\int_0^T \int_\Omega  \widetilde \Psi_\e:\nabla (w_{\eta, \varepsilon}^i\phi)\dd x\dd t
\to \int_0^T \int_\Omega \Psi:\nabla ({\overline  w}^i\varphi )\dd x\dd t.
\ea

Using $\nabla w_{\eta ,\varepsilon }^i \to 0$ and the strong convergence of ${w_{\eta ,\varepsilon }^i}$ again implies
\ba\label{conv-small-4}
\int_0^T \int_\Omega  \nabla \widetilde \UU_\varepsilon:\nabla (w_{\eta, \varepsilon}^i\varphi)\dd x\dd t &
=\int_0^T \int_\Omega \nabla \widetilde \UU_\varepsilon :\nabla w_{\eta, \varepsilon}^i\varphi \dd x\dd t
+\int_0^T \int_\Omega  \nabla \widetilde \UU_\varepsilon : \nabla \varphi  \otimes w_{\eta, \varepsilon}^i\dd x\dd t\\
&\to \int_0^T \int_\Omega  \nabla \UU :\nabla \varphi  \otimes {\overline  w}^i\dd x\dd t
=\int_0^T \int_\Omega \nabla \UU :\nabla ({\overline  w}^i\varphi)\dd x\dd t.
\ea

For the term related to the pressure, again by the strong convergence of ${w_{\eta ,\varepsilon }^i}$, we have
\ba\label{conv-small-5}
\int_0^T \int_\Omega  \widetilde P_\e \dive (w_{\eta,\varepsilon}^i \phi) \dd x\dd t
&=\int_0^T \int_\Omega  \widetilde P_\e ( w_{\eta,\varepsilon}^i \cdot \nabla \phi) \dd x\dd t\\
&\to\int_0^T \int_\Omega  P(\overline w^i\cdot\phi) \dd x\dd t
=\int_0^T \int_\Omega   P \dive(\overline w^i \phi) \dd x\dd t.
\ea

For the external force term, by the strong convergence of $\widetilde \FF_{\e}$ in \eqref{Fe-F-st}, we have
\ba\label{conv-small-6}
\int_0^T \int_\Omega  \widetilde \FF_\e ( w_{\eta,\varepsilon}^i \phi) \dd x\dd t
\to\int_0^T \int_\Omega  \FF (\overline w^i\phi) \dd x\dd t.
\ea

Summarizing the convergences in  \eqref{conv-small-1}--\eqref{conv-small-6} and passing $\e\to 0$ in \eqref{weak-equa-small} implies
\ba\label{equa-U-P-small-weak}
&\int_0^T \int_\Omega  \uu  \cdot (\overline  w^i\phi)-\Psi :\nabla (\overline  w^i \phi)+\mu \nabla \UU:\nabla (\overline  w^i \phi)- P \dive(\overline  w^i \phi )\dd x\dd t \\&=\int_0^T \int_\Omega  \FF\cdot (\overline  w^i \phi )\dd x\dd t+ \int_0^T  \int_\Omega  \uu^0  \cdot (\overline  w^i\phi )  \,\dd x\dd t.
\ea
Equation \eqref{equa-U-P-small-weak} is equivalent to
\ba\nonumber
&\int_0^T \int_\Omega  \uu  \cdot (A\varphi)-\Psi :\nabla (A\varphi)+\mu \nabla \UU:\nabla (A\varphi)- P \dive(A\varphi )\dd x\dd t \\&=\int_0^T \int_\Omega  \FF\cdot (A\varphi )\dd x\dd t+\int_0^T  \int_\Omega \uu^0  \cdot (A\varphi )  \,\dd x\dd t
\ea
for any $\varphi\in C_c^\infty( \Omega \times (0,T) ;\mathbb{R}^3)$ with $A=(A_{i,j}) = (\overline w^{j})_{i} = (\overline w^{i})_{j}.$ As shown in \cite{Allaire91}, $A$ is a positive definite matrix.  This implies
\[\uu - \uu^0+\dive\Psi - \mu \Delta \UU+\nabla P-\FF = 0, \quad \mbox{in} \ \mathcal{D}'(\Omega\times (0,T)).\]
The proof of Proposition \ref{thm-UP-crit} is completed.
\end{proof}

\subsection{End of the proof}
From Proposition \ref{thm-UP-crit}, we can deduce the limit equations in $\uu$. Indeed, differentiating in $t$ to equation \eqref{equa-U-P-small} implies
\ba\label{eq-uu-pf-0}
\partial_t \uu+ \dive(\uu \otimes \uu) - \mu \Delta \uu+ \nabla {p} = \ff, \ \mbox{in} \ \mathcal{D}'(\Omega \times (0,T)),
\ea
with $p = \p_{t} P$.  By $\eqref{U-P-conv}_{2}$, we know that
\ba\nn
\widetilde p_{\e}  : = \p_{t} \widetilde P_{\e} \to p \ \mbox{weakly in} \ W^{-1,2}(0,T;L_{0}^{2}(\Omega)),
\ea
which is exactly \eqref{thm1-conv-p}.

Clearly $\dive \uu = 0$ which follows from $\dive \widetilde \uu_{\e} = 0$. Since $\uu \in L^{\infty}(0,T; L^{2}(\Omega)) \cap L^{2}(0,T;W^{1,2}_{0}(\Omega))$, using equation \eqref{eq-uu-pf-0} implies
\ba\nn
\partial_t \uu \in L^{\frac{4}{3}}(0,T;V^{-1,2}(\Omega)).
\ea
Thus
\ba\label{u-contin-t}
\uu \in C_{weak}([0,T), L^{2}(\Omega))\cap C([0,T), L^{q}(\Omega)),  \ \mbox{for any $1\leqslant q<2$}.
\ea
 We shall further show the attainment of initial datum for $\uu$. With such continuity of $\uu$ in time variable in \eqref{u-contin-t}, by density argument, we deduce from \eqref{eq-uu-pf-0} that for each $\varphi \in C_{c}^{1}(\Omega \times [0,T);\RR^{3}), \ \dive \varphi =0$,
\ba\label{u-weak-1}
\int_{0}^{T} \int_\Omega - \uu \p_{t} \varphi - \uu\otimes \uu : \nabla \varphi  + \nabla \uu :\nabla \varphi \dd x \dd t  = \int_{0}^{T} \int_\Omega \ff \cdot \varphi + \int_{\Omega} \uu(x, 0) \varphi(x,0) \dd x.
\ea

Following the proof of Proposition \ref{moment-equa}, we can deduce for  each $\varphi \in C_{c}^{1}(\Omega \times [0,T);\RR^{3}), \ \dive \varphi =0$ that
\ba\label{u-weak-2}
\int_{0}^{T} \int_\Omega - \uue \p_{t} \varphi - \uue \otimes \uue : \nabla \varphi  + \nabla \uue :\nabla \varphi \dd x \dd t  = \int_{0}^{T} \int_\Omega \widetilde \ff_{\e} \cdot \varphi + \int_{\Omega} \uu_\e^{0}(x) \varphi(x,0) \dd x + {\mathbf H}_{\e}(\varphi),
\ea
where the remainder term ${\mathbf H}_{\e}(\varphi)$ satisfies
$$
|{\mathbf{H}}_{\e}(\varphi)| \leq C \e^{\sigma} \big(\|  \partial _t \varphi \| _{L^{\frac 43} L^{2}}
+\| \nabla \varphi\| _{L^{4} L^{r_1}}  + \| \varphi(\cdot,0)\|_{L^{r_{1}^{*}}(\Omega)}\big),
$$
with $\sigma>0$ and $2<r_{1}<3$ given in Proposition \ref{moment-equa} and $r_{1}^{*}$ the Sobolev conjugate number of $r_{1}$ such that $\frac{1}{r_{1}^{*}} = \frac{1}{r_{1}} - \frac{1}{3}$.

By the convergence of the convective term in \eqref{convective-conv}, passing $\e\to 0$ in \eqref{u-weak-2} implies
\ba\label{u-weak-3}
\int_{0}^{T} \int_\Omega - \uu \p_{t} \varphi - \uu\otimes \uu : \nabla \varphi  + \nabla \uu :\nabla \varphi \dd x \dd t  = \int_{0}^{T} \int_\Omega \ff \cdot \varphi + \int_{\Omega} \uu^{0}(x) \varphi(x,0) \dd x.
\ea
Comparing \eqref{u-weak-1} and \eqref{u-weak-3} implies the attainment of initial datum:
$$
\uu|_{t=0} = \uu^{0}.
$$

\medskip


We obtain system \eqref{equa-u-p-small} and thus complete the proof of  Theorem \ref{thm-1}.

\section{Proof of Theorem \ref{thm-3}}\label{sec-thm3}\label{sec-thm2}

In this section we shall prove Theorem \ref{thm-3} where $a_{\e} = \e^{\alpha}$ \eqref{ae-se} with $1<\alpha <3$ and $\sigma_{\e} = \e^{\frac{3-\alpha}{2}}$ concerning the case of large holes. In this case we consider the time-scaled Navier-Stokes system \eqref{1-time-equa}.

In the case of large holes, one can benefit from the zero boundary condition on the holes and obtain the following perforation version of Poincar\'e inequality (see Lemma 3.4.1 in \cite{ref4}):
\begin{lemma}\label{lem-Poincare}
Let $\Omega_{\e}$ be the perforated domain defined by \eqref{1-hole} and \eqref{1-domain} with  $a_{\e} = \e^{\alpha}, \ 1<\alpha <3$. There exists a constant $C$ independent of $\varepsilon$ such that
\begin{equation}
\|u\|_{{L^2}({\Omega _\varepsilon })} \leqslant C\sigma_\varepsilon \|\nabla u\|_{{L^2}({\Omega _\varepsilon })}, \quad \mbox{for any $u \in W_0^{1,2}({\Omega _\varepsilon })$.}
\label{11}
\end{equation}

\end{lemma}

\subsection{Estimates of velocity}
From the energy inequality \eqref{6}, using Poincar\'e inequality \eqref{11} and  H\"older's inequality gives
\ba\nn
&\frac{1}{2}\sigma _\varepsilon^2\int_{\Omega_\varepsilon} \left| {\uu}_\varepsilon(x,t) \right|^2 \dd x
+ \int_0^t \int_{\Omega _\varepsilon} \left| \nabla \uu_\varepsilon (x,s) \right|^2 \dd x\dd s\\
&\leqslant \int_0^t \int_{\Omega _\varepsilon} \ff_{\e} \cdot \uu_\varepsilon  \dd x\dd s
+ \frac{1}{2}\sigma _\varepsilon^2 \int_{\Omega_\varepsilon} \left| {\uu}^0_{\e}(x) \right|^2 \dd x\\
&\leqslant\| \ff_{\e}\|  _{L^{2}(0,t; L^2(\Omega_\varepsilon))}
\|  {\uu}_\varepsilon \| _{L^{2}(0,t; L^2(\Omega_\varepsilon ))}
+ \frac{1}{2}\sigma_\varepsilon^2\sup_{0<\e\leqslant 1} \|\uu_{\e}^{0}\|_{L^{2}(\Omega_{\e})}^{2} \\
&\leqslant C \sigma _\e\| \ff_{\e}\|  _{L^{2}(0,t; L^2(\Omega _\varepsilon))}
\|  \nabla \uu_\varepsilon\| _{L^{2}(0,t; L^2(\Omega_\varepsilon))}
+\frac{1}{2}\sigma _\varepsilon^2 \sup_{0<\e\leqslant 1} \|\uu_{\e}^{0}\|_{L^{2}(\Omega_{\e})}^{2}\\
& \leqslant C \sigma _\varepsilon^2 \sup_{0<\e\leqslant 1}\| \ff_{\e}\|  _{L^{2}(0,T; {L^2}({\Omega _\varepsilon }))}^{2} + \frac{1}{2}\| {{{\nabla \uu}_\varepsilon }}\| _{L^{2}(0,t; {L^2}({\Omega _\varepsilon }))}^{2} + \frac{1}{2}\sigma _\varepsilon^2\sup_{0<\e\leqslant1} \|\uu_{\e}^{0}\|_{L^{2}(\Omega_{\e})}^{2}.
\ea
This implies
\begin{equation}\label{est-large-ue}
\|  {\uu}_\varepsilon \| _{L^{\infty}(0,T;L^2(\Omega _\varepsilon))} \leqslant C, \quad\| \nabla \uu_\varepsilon \| _{L^2 (0,T;L^2(\Omega _\varepsilon))}
\leqslant C \sigma_\e,\quad\|  \uu_\varepsilon \| _{L^2 (0,T;L^2(\Omega_\varepsilon))}
\leqslant C\sigma_\e^2,
\end{equation}
where we used again \eqref{11}.  Therefore the extension satisfies
\begin{equation}\label{3-est-u}
\|  \uue\| _{L^{\infty}(0,T;L^2(\Omega))} \leqslant C, \quad\| \nabla \uue \| _{L^2 (0,T;L^2(\Omega))}
\leqslant C \sigma_\e,\quad\|  \uue \| _{L^2 (0,T;L^2(\Omega))}
\leqslant C \sigma_\e^2.
\end{equation}
Then, up to a subsequence, we have the convergence
\ba\label{UUe-conv-large}
\sigma_\e^{-2} \uue\to \uu \ \mbox{weakly in} \ L^{2}(0,T; L^2(\Omega)),
\ea
which is exactly \eqref{thm2-u}.

By \eqref{1-new-def}, it follows from \eqref{3-est-u} that
\begin{equation}\label{3-estimate-U}
 \|\widetilde \UU_\varepsilon  \|_{W^{1,\infty}(0,T; L^2(\Omega))}\leqslant C,\quad \|  \nabla \widetilde \UU_\varepsilon  \|_{W^{1,2}(0,T; L^{2} (\Omega))}\leqslant C \sigma_\e,\quad
\| \widetilde \UU_\varepsilon  \|_{W^{1,2}(0,T; L^{2} (\Omega))}\leqslant C \sigma_\e^{2}.
\end{equation}
Then, up to a subsequence,
\ba\label{UUe-conv-large}
\sigma_\e^{-2} \widetilde \UU_\varepsilon \to \UU \ \mbox{weakly in} \ W^{1,2}(0,T; L^2(\Omega)).
\ea
As the argument in \eqref{U-u-small}, the uniqueness of the weak limits implies that
\ba\label{U-u-large}
\UU(\cdot ,t) = \int_0^t \uu (\cdot,s)\dd s.
\ea

\subsection{Extension of pressure}
Let $P_{\e}$  be given in the Stokes equations \eqref{2-equa-large-holes}. Recall the estimate of $R_\varepsilon$ in \eqref{pt-res-2}:
\[
\|\nabla R_\e(\varphi)\|_{L^q (0,T; L^{2}(\Omega_\e))} \leqslant C \, \big( \|\nabla \varphi\|_{L^q (0,T; L^{2}(\Omega_\e))} + \sigma_{\e}^{-1} \|\varphi\|_{L^q (0,T; L^{2}(\Omega_\e))}\big),
\]
which will be repeatedly used in this section. Here  $\sigma_{\e} \to 0$ for the case of large holes. As in Section \ref{sec:extP-small}, we define a functional $\widetilde H_\e$ in $\mathcal{D}'(\Omega\times(0,T);\RR^{3})$ by the following dual formulation: for each $\varphi \in C_c^\infty(\Omega\times(0,T);\RR^{3}),$
\ba\label{def-large-F}
\l   \widetilde H_\e ,\varphi \r_{\Omega\times(0,T)}&
=\l\nabla {P_\varepsilon}(t), R_\varepsilon (\varphi) \r_{\Omega_{\e}\times(0,T)}\\
&=\l\FF_\e(t)-\sigma^2_\e\uu_\varepsilon(t)+\sigma^2_\e\uu^0_\e + \mu \Delta \UU_\varepsilon(t)
- \dive  \Psi_\varepsilon (t),{R_\varepsilon}(\varphi) \r_{\Omega_{\e}\times(0,T)}.
\ea

Using the uniform estimates in \eqref{est-large-ue}--\eqref{3-estimate-U} implies
\ba\label{est-P-large-1}
\left|  \l \dive  \Psi_\varepsilon(t), R_\varepsilon(\varphi)\r_{\Omega_\varepsilon\times(0,T)} \right| &
= \left|  \l \Psi_\varepsilon(t), \nabla R_\varepsilon(\varphi)\r_{\Omega_\varepsilon\times(0,T)} \right| \\
&\leqslant\|   \Psi_{\e}\| _{L^\infty L^2}
\|  \nabla R_\varepsilon (\varphi)\| _{L^2 L^2}\\
&\leqslant\|    \uu_\varepsilon  \otimes \uu_\varepsilon\| _{L^1 L^2}
\|  \nabla R_\varepsilon (\varphi)\| _{L^2 L^2}\\
&\leqslant C\|    \uu_\varepsilon \| _{L^2 L^6}^2(\|  \nabla \varphi \| _{L^2 L^2}
+ \sigma_\e^{-1} \| \varphi \| _{L^2 L^2})\\
&=C\sigma_\varepsilon^2\| \nabla \varphi \| _{L^2 L^2}
+ C\sigma_\varepsilon\| \varphi \| _{L^2 L^2}.
\ea
and
\ba\label{est-P-large-2}
\left|  \l \Delta \UU_\varepsilon(t), R_\varepsilon (\varphi) \r_{\Omega_\varepsilon\times(0,T)} \right|&
\leqslant\| \nabla \UU_\varepsilon(t)\| _{L^2 L^2}
\|  \nabla R_\varepsilon (\varphi)\| _{L^2 L^2}\\
&\leqslant C{\sigma _\varepsilon }(\|  \nabla \varphi \| _{L^2 L^2}
+ \sigma_\e^{-1} \| \varphi \| _{L^2 L^2})\\
&=C\sigma _\varepsilon\| \nabla \varphi \| _{L^2 L^2}
+ C\|  \varphi \| _{L^2 L^2}.
\ea
By Poinca\'re inequality in Lemma \ref{lem-Poincare}, we obtain
\ba\label{est-P-large-3}
\left| \l \sigma^2_\e{\uu}_\e, R_\varepsilon (\varphi) \r_{\Omega _\varepsilon\times(0,T)} \right|
&\leqslant C\sigma^2_\e \| \uu_\e\| _{L^\infty(0,T; L^2(\Omega _\varepsilon))} \sigma_\e
\|  \nabla R_\varepsilon (\varphi) \| _{L^2 L^2} \\
&\leqslant C(\sigma_\e^3\|  \nabla \varphi\| _{L^2 L^2}+\sigma_\e^2\|  \varphi\| _{L^2 L^2}),\\
\left| \l \sigma^2_\e{\uu}^0_\e, R_\varepsilon (\varphi) \r_{\Omega _\varepsilon\times(0,T)} \right|
&\leqslant C\sigma^2_\e \| \uu^0_\e\| _{L^2} \sigma_\e
\|  \nabla R_\varepsilon (\varphi) \| _{L^2 L^2}\\
&\leqslant C(\sigma_\e^3\|  \nabla \varphi\| _{L^2 L^2}+\sigma_\e^2\|  \varphi\| _{L^2 L^2}),\\
\left| \l \FF_\e (t), R_\varepsilon (\varphi) \r_{\Omega _\varepsilon\times(0,T)} \right|
&\leqslant C \sigma_\e\| \FF_\e(t)\| _{L^2 L^2}\| \nabla R_\varepsilon (\varphi) \| _{L^2 L^2} \\
&\leqslant C(\sigma_\e\|  \nabla \varphi\| _{L^2 L^2}+\|  \varphi\| _{L^2 L^2}).
\ea
From \eqref{est-P-large-1}--\eqref{est-P-large-3} we deduce that
\ba\label{est-large-F}
\l \widetilde H_\e,\varphi \r_{\Omega\times(0,T)}
&\leqslant C (\| \varphi \| _{L^2 L^2}+ \sigma_\varepsilon\| \nabla \varphi\| _{L^2 L^2}).
\ea
This indicates $\widetilde H_{\e}$ is bounded in $L^{2}(0,T, L^{2}(\Omega)) + \sigma_{\e} L^{2}(0,T, W^{-1,2}(\Omega)) $.

On the other hand,
\ba
\l \partial_t \widetilde H_\e, \varphi \r_{\Omega\times(0,T)}&
=-\l \widetilde H_\e, \partial_t\varphi \r_{\Omega\times(0,T)}
=-\l \nabla  P_\varepsilon, R_\varepsilon (\partial_t \varphi) \r_{\Omega_\e\times(0,T)}\\
&=-\l\FF_\e(t)-\sigma^2_\e\uu_\varepsilon(t)+\sigma^2_\e\uu^0_\e + \mu \Delta \UU_\varepsilon(t)
-{\Phi_\varepsilon}(t), R_\varepsilon (\partial_t \varphi) \r_{\Omega_\e\times(0,T)}\\
&=\l \ff_\e, R_\varepsilon (\varphi) \r + \l\sigma^2_\e \uu_\varepsilon, R_\varepsilon (\partial_t \varphi) \r
+\l \mu \Delta \uu_\varepsilon, R_\varepsilon (\varphi) \r-\l \dive  (\uu_\e\otimes\uu_\e), R_\varepsilon (\varphi)\r.
\nn\ea
By  the uniform estimates in \eqref{est-large-ue}--\eqref{3-estimate-U} and Lemma \ref{lem-Poincare}, we have
\ba
\left|\l \ff_\e ,R_\varepsilon (\varphi) \r \right| & \leqslant C \| \ff_\e\| _{L^2 L^2}\|  R_\varepsilon (\varphi)\| _{L^2 L^2}
\leqslant C (\sigma_\e\| \nabla \varphi\| _{L^2 L^2}
+\| \varphi\| _{L^2 L^2}), \\
\left|\l\sigma^2_\e \uu_\varepsilon, R_\varepsilon (\partial_t \varphi)\r\right|
& \leqslant C \sigma_\e^2 \sigma_\e\| \nabla R_\varepsilon (\partial_t\varphi)\| _{L^2 L^2}
\leqslant C\sigma_\e^2 (\sigma_\e\| \nabla \partial_t\varphi\| _{L^2 L^2}
+\| \partial_t\varphi\| _{L^2 L^2}),\\
| \l \mu \Delta \uu_\varepsilon, R_\varepsilon (\varphi) \r| & \leqslant C\| \nabla \uu_\varepsilon\| _{L^2 L^2}
\|  \nabla R_\varepsilon (\varphi)\| _{L^2 L^2} \leqslant C (\sigma_\e\| \nabla \varphi\| _{L^2 L^2}
+\| \varphi\| _{L^2 L^2}),\\
| \l \dive  (\uu_\e\otimes\uu_\e), R_\varepsilon (\varphi)\r | &\leqslant C
\|  \widetilde \uu_\varepsilon \otimes\widetilde \uu_\varepsilon\| _{L^{\frac43} L^2}
\|  \nabla R_\varepsilon (\varphi)\| _{L^4 L^2} \leqslant C\sigma_\e^{\frac32}
(\|  \nabla \varphi\| _{L^4 L^2}
+ \sigma_\e^{-1}\|  \varphi\| _{L^4 L^2}).
\nn\ea
This shows that
\ba\nn
\left| \l \partial_t \widetilde H_\e, \varphi \r_{\Omega\times(0,T)}\right|
&\leqslant C (\|   \varphi\| _{L^2 L^2}+\sigma_\e\|  \nabla \varphi\| _{L^2 L^2}
+\sigma_\e^{\frac12}\| \varphi\| _{L^4 L^2}+\sigma_\e^{\frac32}\| \nabla \varphi\| _{L^4 L^2}\\
&\quad\quad\quad\quad\quad\quad\sigma_\e^2\| \partial_t\varphi\| _{L^2 L^2}+\sigma_\e^3\| \nabla \partial_t\varphi\| _{L^2 L^2})\\
&\leqslant C\|  \varphi\| _{L^2 L^2}+C\sigma_\e^{\frac12}\| \nabla \varphi\| _{L^4 L^2}
+C\sigma_\e^2\| \nabla \partial_t\varphi\| _{L^2 L^2}.
\ea
Therefore, $\widetilde H_\e$ is uniformly bounded in $W^{1,2} L^2 +\sigma_\e^{\frac12}  W^{1,\frac{4}{3}} W^{-1,2} + \sigma_\e^2 L^2 W^{-1,2}$. Moreover, the second property of the restriction operator in \eqref{pt-res} implies that
\ba
\l \widetilde H_\e, \varphi \r_{\Omega\times(0,T)}=0,\ \forall  \varphi\in C_{c}^{\infty}(\Omega \times (0,T)  ;\RR^{3}), \ \dive  \varphi=0.
\nn\ea
Thus, we can apply Lemma \ref{lemma-large} and deduce that there exists $\widetilde P_\e \in W^{1,2} (W^{1,2}\cap L_{0}^{2}) + \sigma_\e^{\frac12} W^{1,\frac43} L^2_0+\sigma_\e^2 L^2 L^2_0 $ such that
\ba\label{large-F-P-1}
\widetilde H_\e =\widetilde H_\e^{(1)} + \sigma_\e^{\frac12} \widetilde H_\e^{(2)}+\sigma_\e^2 \widetilde H_\e^{(3)},\quad
\widetilde P_\e =\widetilde P_\e^{(1)} + \sigma_\e^{\frac12} \widetilde P_\e^{(2)}+\sigma_\e^2 \widetilde P_\e^{(3)},
\ea
with
\ba\label{large-F-P-2}
\widetilde H_\e^{(1)} =\nabla \widetilde P_\e^{(1)},\quad
\widetilde H_\e^{(2)} =\nabla \widetilde P_\e^{(2)},\quad
\widetilde H_\e^{(3)} =\nabla \widetilde P_\e^{(3)},
\ea
and
\ba\label{est-large-P-2}
& \|   \widetilde P_\e^{(1)} \|_{W^{1,2} W^{1,2}}\leqslant \| \widetilde H_\e^{(1)}  \|_{W^{1,2} L^2} \leqslant C,\\
& \|   \widetilde P_\e^{(2)} \|_{W^{1,\frac43} L^2_0}\leqslant  \| \widetilde H_\e^{(2)}  \|_{W^{1,\frac43} W^{-1,2}} \leqslant C,\\
& \| \widetilde P_\e^{(3)} \|_{L^2 L^2_0}\leqslant  \| \widetilde H_\e^{(3)}  \|_{L^2 W^{-1,2}}\leqslant C.
\ea


\subsection{Passing to the limit}

Instead of showing the limit equations in $\uu$ directly, we prove the following results in $\UU$:
\begin{proposition}\label{prop-UP-large}
Let $(\UU_\varepsilon, P_\varepsilon)$ be the weak solution of equation \eqref{2-equa-large-holes} and $(\widetilde \UU_\varepsilon, \widetilde P_\varepsilon)$ be their extension in $\Omega$ defined in \eqref{1-new-def}, \eqref{def-large-F}, \eqref{large-F-P-1} and \eqref{large-F-P-2}. If ${\lim _{\varepsilon  \to 0}}{\sigma _\varepsilon } = 0~i.e. ~1<\alpha<3$,  then
\ba\label{conv-large-U}
\sigma_\e^{-2}\widetilde \UU_\varepsilon \to \UU\ \mbox{weakly\ in} \ L^2 (0,T;L^2(\Omega;\mathbb{R}^3)),
\ea
and
\ba\label{conv-large-P-1}
\widetilde P_\varepsilon=\widetilde P_\varepsilon^{(1)}+\sigma_\e^{\frac{1}{2}}\widetilde P_\varepsilon^{(2)}+\sigma_\e^2 \widetilde P_\varepsilon^{(3)},
\ea
which satisfies the estimates in \eqref{est-large-P-2} and
\ba\label{conv-large-P-2}
\widetilde P_\varepsilon^{(1)}  \to P\ \mbox{weakly\ in} \ W^{1,2}(0,T;W^{1,2}(\Omega)).
\ea

Moreover, $(\UU,P)$ solves
\begin{equation}\label{equa-U-P-large}
\mu \UU = A(\FF- \nabla P),\ \mbox{in}\  \mathcal{D}'(\Omega\times (0,T)),
\end{equation}
with
\ba\label{U-F}
\UU(\cdot ,t) = \int_0^t \uu (\cdot,s)\dd s, \quad \FF(\cdot ,t) = \int_0^t \ff (\cdot,s)\dd s.
\ea
\end{proposition}
\begin{proof}
The  convergence of $\widetilde \UU_{\e}$ in \eqref{conv-large-U} follows immediately from the uniform estimates \eqref{3-estimate-U} and was given already in \eqref{UUe-conv-large}.  The weak convergence of $\widetilde P_{\e}$ in \eqref{conv-large-P-2}  follows from the uniform estimates \eqref{est-large-P-2}.  The relations in \eqref{U-F} follows from the uniqueness fo weak limits and were given already in \eqref{Fe-F-st} and \eqref{U-u-large}.

Given any scalar test function $\phi \in C_{c}^{\infty}(\Omega \times (0,T) )$, we choose $w ^i_{\eta,\varepsilon}\phi$ as a test function to \eqref{2-equa-large-holes} and obtain
\ba\label{weak-equa-large}
&\int_0^T \int_\Omega  \sigma_\e^2\widetilde \uu_\varepsilon  \cdot (w_{\eta, \varepsilon}^i\phi)
-\widetilde \Psi _\varepsilon:\nabla (w_{\eta, \varepsilon}^i\phi)
+\mu \nabla \UU_\varepsilon :\nabla (w_{\eta ,\varepsilon }^i\phi)
-\widetilde  P_\varepsilon\dive (w_{\eta, \varepsilon}^i\phi )\dd x\dd t \\
&=\int_0^T \int_\Omega  \FF_\e \cdot (w_{\eta,\varepsilon}^i\phi )\dd x\dd t+\sigma_\e^2 \int_0^T  \int_\Omega \widetilde \uu_\e^0  \cdot (w_{\eta,\varepsilon}^i\phi )   \,\dd x\dd t.
\ea
By the uniform estimates in \eqref{3-est-u} and \eqref{lem-cell-2}, and the convergences in \eqref{cell-3} and \eqref{Fe-F-st}, we have
\ba\label{conv-large-1}
\left| \sigma_\e^2\int_0^T \int_\Omega  \widetilde \uu_\varepsilon  \cdot (w_{\eta,\varepsilon}^i\phi)\dd x\dd t \right|
&\leqslant \sigma_\e^2\|  \widetilde \uu_\varepsilon \| _{L^2 L^6}\| w_{\eta, \varepsilon}^i\| _{L^2(\Omega)}
\|  \phi \| _{L^2 L^3}  \leqslant C\sigma _\varepsilon^3 \to 0,\\
\left| \sigma_\e^2\int_0^T \int_\Omega  \uu^0_\e \cdot (w_{\eta,\varepsilon}^i\phi )\dd x\dd t \right|
& \leqslant C \sigma_\e^2   ||\uu^0||_{L^2(\Omega)} ||w_{\eta,\varepsilon}^i||_{L^2(\Omega)}\| \phi \| _{L^1 L^\infty}  \leqslant C\sigma _\varepsilon^2 \to 0 \to 0,\\
\int_0^T \int_\Omega  \widetilde \FF_\e \cdot (w_{\eta,\varepsilon}^i\phi) \dd x\dd t
& \to \int_0^T \int_\Omega  \FF \cdot ({\overline w^i}\phi) \dd x\dd t.
\ea

For the convective term, by \eqref{lem-cell-2} and  \eqref{3-est-u}, we have
\ba\label{conv-large-4}
&\left| \int_0^T \int_\Omega  \widetilde \Psi _\varepsilon:\nabla (w_{\eta, \varepsilon}^i\phi) \dd x\dd t \right| \leqslant \| \widetilde \Psi_{\e} \|_{L^{\infty} L^{2}}\| w_{\eta, \varepsilon}^i\| _{W^{1,2}}\| \phi \| _{W^{1,\infty}}\\
&  \leqslant \| \uue\otimes \uue\| _{L^1 L^2}\| w_{\eta, \varepsilon}^i\| _{W^{1,2}}\| \phi \| _{W^{1,\infty}}\leqslant  C\| \uue\| _{L^2 L^6}^{2}\| w_{\eta, \varepsilon}^i\| _{W^{1,2}} \leqslant C \sigma_{\e} \to 0.
\ea

Using equation \eqref{cell-general} in $(w_{\eta,\varepsilon}^i,q_{\eta,\varepsilon}^i)$ in Section \ref{sec:cell} gives us
\ba\label{conv-large-2}
&\int_0^T \int_\Omega  \nabla \widetilde \UU_\varepsilon :\nabla (w_{\eta, \varepsilon}^i\phi )\dd x\dd t
= \int_0^T \int_\Omega  \nabla \widetilde \UU_\varepsilon :\nabla \phi  \otimes w_{\eta,\varepsilon}^i
+ \nabla \widetilde \UU_\varepsilon :\nabla w_{\eta,\varepsilon}^i\phi \dd x\dd t\\
&= \int_0^T \int_\Omega  \nabla \widetilde \UU_\varepsilon  :\nabla \phi  \otimes w_{\eta,\varepsilon}^i
+ \nabla(\phi \widetilde \UU_\varepsilon ):\nabla w_{\eta,\varepsilon}^i
- \nabla \phi  \otimes \widetilde \UU_\varepsilon :\nabla w_{\eta,\varepsilon}^i \dd x\dd t\\
&= \int_0^T \int_\Omega  \nabla \widetilde \UU_\varepsilon :\nabla \phi  \otimes w_{\eta,\varepsilon}^i
+ \varepsilon^{-1} \dive  (\phi \widetilde \UU_\varepsilon) q_{\eta,\varepsilon}^i
+\sigma_\varepsilon^{-2}(\phi \widetilde \UU_\varepsilon) \cdot e^i
- \nabla \phi  \otimes \widetilde \UU_\varepsilon :\nabla w_{\eta,\varepsilon}^i\dd x\dd t.
\ea
For the first term and the last term on the right-hand side of \eqref{conv-large-2}, we have
\ba\nonumber
\left| \int_0^T \int_\Omega  \nabla \widetilde \UU_\varepsilon :\nabla \phi  \otimes w_{\eta,\varepsilon}^i\dd x\dd t \right|
& \leqslant  \| \nabla \widetilde \UU_\varepsilon  \|_{L^2 L^2}
\|  w_{\eta,\varepsilon}^i\| _{L^2 (\Omega)}
 \| \nabla \phi   \|_{L^2 L^\infty(\Omega)} \leqslant C{\sigma_\varepsilon} \to 0, \\
\left| \int_0^T \int_\Omega \nabla \phi  \otimes \widetilde \UU_\varepsilon :\nabla w_{\eta,\varepsilon}^i \dd x\dd t \right|&
\leqslant\| \widetilde \UU_\varepsilon \| _{L^2 L^2}
\|  \nabla w_{\eta,\varepsilon}^i\| _{L^2(\Omega)}\| \nabla \phi \| _{L^2 L^\infty(\Omega)}\leqslant C{\sigma_\varepsilon} \to 0.
\ea
Using $\dive  \widetilde \UU_\varepsilon = 0$ and $\varepsilon^{-1} c_\eta= \sigma_\varepsilon^{-1}$ implies
\ba\label{conv-large-6}
\left| \int_0^T \int_\Omega  \varepsilon ^{-1}\dive (\phi \widetilde \UU_\varepsilon )q_{\eta,\varepsilon}^i
\dd x\dd t \right| &\leqslant \varepsilon ^{-1}\|  \nabla \phi \| _{L^2 L^\infty (\Omega)}\|  \widetilde \UU_\varepsilon \| _{L^2 L^2}\| q_{\eta,\varepsilon}^i\| _{L^2(\Omega)}\\
&\leqslant C \varepsilon ^{-1}\sigma _\varepsilon ^2 c_\eta  = C\sigma _\varepsilon \to 0.
\ea
By the weak convergence of $\sigma_\e^{-2} \widetilde \UU_\e$ in \eqref{UUe-conv-large}, we have
\ba\label{conv-large-7}
\int_0^T \int_\Omega \sigma_\e^{-2} \phi \widetilde \UU_\varepsilon\cdot {e^i} \dd x\dd t \to \int_0^T \int_\Omega  \phi \UU \cdot {e^i} \dd x\dd t.
\ea

Now we deal with the term related to the pressure. By the decomposition in \eqref{large-F-P-1}--\eqref{est-large-P-2} we have
\ba\label{large-P-1}
& \int_0^T \int_\Omega \widetilde P_\varepsilon \dive (w_{\eta, \varepsilon}^i \phi ) \dd x\dd t
=\int_0^T \int_\Omega \widetilde P_\varepsilon^{(1)} w_{\eta, \varepsilon}^i \cdot \nabla \phi \dd x\dd t \\
&\qquad +\sigma_\e^{\frac{1}{2}}\int_0^T \int_\Omega  \widetilde P_\varepsilon^{(2)} w_{\eta, \varepsilon}^i  \cdot \nabla \phi  \dd x\dd t +\sigma_\e^2 \int_0^T \int_\Omega  \widetilde P_\varepsilon^{(3)} w_{\eta, \varepsilon}^i  \cdot \nabla \phi  \dd x\dd t.
\ea
By the uniform estimates in \eqref{est-large-P-2}, there holds
\ba\label{large-P-2}
\sigma_\e^{\frac{1}{2}}\int_0^T \int_\Omega  \widetilde P_\varepsilon^{(2)} w_{\eta, \varepsilon}^i  \cdot \nabla \phi  \dd x\dd t
\leqslant  \sigma_\e^{\frac{1}{2}}\| \widetilde P_\varepsilon^{(2)}\| _{L^{\frac43} L^2}
\|  w_{\eta, \varepsilon}^i\| _{L^2}\| \nabla \phi\| _{L^{4} L^\infty}
\leqslant C\sigma_\e^{\frac{1}{2}} \to 0, \\
\sigma_\e^2 \int_0^T \int_\Omega  \widetilde P_\varepsilon^{(3)} w_{\eta, \varepsilon}^i  \cdot \nabla \phi  \dd x\dd t
\leqslant  \sigma_\e^2\| \widetilde P_\varepsilon^{(3)}\| _{L^2 L^2}
\|  w_{\eta, \varepsilon}^i\| _{L^2}\| \nabla \phi\| _{L^2 L^\infty}
\leqslant C\sigma_\e^2\to 0.
\ea
Since $\widetilde P_\varepsilon^{(1)}\to P$ weakly in $W^{1,2}(0,T;W^{1,2}(\Omega))$, by Sobolev compact embedding, one has $\widetilde P_\varepsilon^{(1)}\to P$ strongly in $L^2(0,T;L^2_0(\Omega))$. Thus,
\ba\label{large-P-3}
\int_0^T \int_\Omega \widetilde P_\varepsilon^{(1)} w_{\eta, \varepsilon}^i \cdot \nabla \phi \dd x\dd t
\to \int_0^T \int_\Omega P (\overline  w^i \cdot \nabla \phi) \dd x\dd t
=\int_0^T \int_\Omega  P \dive (\overline  w^i \phi ) \dd x\dd t.
\ea

From \eqref{large-P-1}--\eqref{large-P-3} we deduce that
\ba\label{conv-large-8}
\int_0^T \int_\Omega  \widetilde P_\varepsilon \dive  (w_{\eta,\varepsilon}^i\phi ) \dd x\dd t
\to \int_0^T \int_\Omega  P{\dive }(\overline w ^i\phi ) \dd x\dd t.
\ea

By \eqref{conv-large-1}--\eqref{conv-large-8}, passing $\e\to 0$ in \eqref{weak-equa-large} implies our desired equation \eqref{equa-U-P-large}.  The proof of Proposition \ref{prop-UP-large} is completed.

\end{proof}

\subsection{End of the proof}
By Proposition \ref{prop-UP-large}, differentiating in $t$ to equation \eqref{equa-U-P-large} implies
\ba\nn
\mu \uu = A({\ff} - \nabla p), \ \mbox{in} \ \mathcal{D}'(\Omega \times (0,T)),
\ea
with $p = \p_{t} P$.  By \eqref{est-large-P-2}, \eqref{conv-large-P-1} and \eqref{conv-large-P-2}, we know that
\ba\nn
\widetilde p_\varepsilon=\widetilde p_\varepsilon^{(1)}+\sigma_\e^{\frac12}\widetilde p_\varepsilon^{(2)}+\sigma_\e^2\widetilde p_\varepsilon^{(3)},
\ea
with
\ba\nn
&\widetilde p_\varepsilon^{(1)}  \to p \ \mbox{weakly in} \ L^2 (0,T;W^{1,2}(\Omega)),\\
&\widetilde p_\varepsilon^{(2)} \ \mbox{bounded in} \ L^\frac{4}{3}(0,T;L_0^2(\Omega)),\\
&\widetilde p_\varepsilon^{(3)}  \ \mbox{bounded in} \ W^{-1,2} (0,T;L_0^2(\Omega)),
\ea
which are exactly \eqref{thm2-p-1} and \eqref{thm2-p-2}.

Since $\widetilde \uu_\varepsilon \in L^2(0,T;W_0^{1,2}(\Omega ))$, $\dive \widetilde \uu_\varepsilon=0$ and $ \sigma_\e^{-2} \widetilde \uu_\varepsilon \to \uu$ weakly in $L^2(0,T;L^2(\Omega))$, then $\dive \uu=0$ in $\Omega\times (0,T) $ and $\uu\cdot \mathbf{n}=0$ on $\partial \Omega \times (0,T)$ (see  \cite{ref20}). We thus complete the proof of  Theorem \ref{thm-3}.


\section*{Acknowledgments}

Yong Lu has been supported by the Recruitment Program of Global Experts of China. Both authors are partially supported by the  NSF of China under Grant 12171235. On behalf of all authors, the corresponding author states that there is no conflict of interest.

\end{document}